\documentclass[11pt]{article}
\usepackage{amsfonts}
\usepackage{color}
\usepackage{amsmath,amssymb,comment}
\usepackage{verbatim}
\newtheorem{theo}{Theorem}[section]
\newtheorem{lem}[theo]{Lemma}
\newtheorem{cor}[theo]{Corollary}

\newcommand{\mysection}[1]{\section{#1} \setcounter{equation}{0}}
    \makeatletter
\def\@fnsymbol#1{\ensuremath{\ifcase#1\or *\or \ddagger\or
   \mathsection\or \mathparagraph\or \|\or **\or \dagger\dagger
   \or \ddagger\ddagger \else\@ctrerr\fi}}
    \makeatother
\newcommand{\proof}{{\sc Proof.} \quad}
\newcommand{\proofc}{{\sc Proof} \ }
\newcommand{\be}{\begin{equation} \label}
\newcommand{\ee}{\end{equation}}
\newcommand{\bea}{\begin{eqnarray}\label}
\newcommand{\eea}{\end{eqnarray}}
\newcommand{\bas}{\begin{eqnarray*}}
\newcommand{\eas}{\end{eqnarray*}}
\newcommand{\bit}{\begin{itemize}}
\newcommand{\eit}{\end{itemize}}
\newcommand{\qed}{\hfill$\Box$ \vskip.2cm}
\newcommand{\nn}{\nonumber}
\newcommand{\R}{\mathbb{R}}

\newcommand{\N}{\mathbb{N}}
\newcommand{\pO}{\partial\Omega}

\newcommand{\eps}{\varepsilon}

\newcommand{\wto}{\rightharpoonup}

\newcommand{\hra}{\hookrightarrow}
\newcommand{\io}{\int_\Omega}

\newcommand{\na}{\nabla}
\newcommand{\Del}{\Delta}
\newcommand{\del}{\delta}
\newcommand{\al}{\alpha}

\newcommand{\Lam}{\Lambda}

\newcommand{\pa}{\partial}
\newcommand{\bom}{\overline{\Omega}}
\newcommand{\Om}{\Omega}

\newcommand{\wh}{\widehat}

\newcommand{\hs}{\hspace*}
\newcommand{\sm}{\setminus}

\newcommand{\vp}{\varphi}
\newcommand{\lbal}{\left\{ \begin{array}{l}}
\newcommand{\lball}{\left\{ \begin{array}{ll}}
\newcommand{\ear}{\end{array} \right.}

\newcommand{\abs}{\\[5pt]}

\newcommand{\adb}{\allowdisplaybreaks}

\newcommand{\tme}{T_{max,\eps}}

\newcommand{\ueps}{u_\eps}
\newcommand{\veps}{v_\eps}
\newcommand{\weps}{w_\eps}
\newcommand{\heps}{h_\eps}
\newcommand{\feps}{f_\eps}

\newcommand{\yeps}{y_\eps}

\begin{document}
\adb
\title{Preventing $L^p$ blow-up by local anisotropy of signal production\\
in the Keller-Segel system with strongly differing diffusion rates}
\author
{
Youshan Tao\footnote{taoys@sjtu.edu.cn}\\
{\small School of Mathematical Sciences, CMA-Shanghai, Shanghai Jiao Tong University,}\\
{\small Shanghai 200240, P.R.~China}
 \and
Michael Winkler\footnote{michael.winkler@math.uni-paderborn.de}\\
{\small Universit\"at Paderborn, Institut f\"ur Mathematik,}\\
{\small 33098 Paderborn, Germany} }
\date{}
\maketitle
\begin{abstract}
\noindent
In a smoothly bounded domain $\Om\subset\R^n$, $n\le 5$, the manuscript considers the variant of the Keller-Segel system given by
\bas
	\left\{ \begin{array}{l}	
	u_t = D \Delta u - \nabla \cdot (u\nabla v), \\[1mm]
	v_t = d \Delta v + \nabla \cdot (u\nabla v) - v + u,
	\end{array} \right.
\eas
which involves an additional contribution $\na \cdot (u\na v)$ to the chemoattractant evolution,
in line with refined modeling literature reflecting
an anisotropic correction to the isotropic signal production term $+u$ in the classical Keller-Segel model.\abs
It is shown that for arbitrary $D>0$ and $d>0$ and
any nonnegative intial data from $W^{1, \infty}(\Om)\times W^{1, \infty}(\Om)$,
an associated Neumann problem admits a global weak solution $(u,v)$ which, inter alia, satisfies
\bas
	{\rm ess} \sup_{\hspace*{-2.5mm} t>0} \io e^{u^\alpha(\cdot,t)} < \infty
\eas
with some $\alpha>0$.\abs

\end{abstract}
\newpage
\section{Introduction}\label{intro}
To describe aggregation phenomena in bacterial populations, Keller and Segel (\cite{KS}) proposed the model
\be{00}
	\left\{ \begin{array}{l}	
	u_t = D \Del u - \na \cdot (u\na v), \\[1mm]
	v_t =  d \Del v - v + u,
	\end{array} \right.
\ee
with $u=u(x,t)$ and $v=v(x,t)$ denoting the population density and signal concentration, respectively,
in which the bacteria are attracted by a chemical signal produced by themselves. Chemotaxis mechanisms
of this form have been found to have wide applications in biology and ecology, and aslo in social sciences
(cf., e.g., \cite{painter_JTB}).\abs
This model (\ref{00}) possesses two favorable mathematical properties:
Firstly, it has a Lyapunov structure formally expressed in the identity
\be{001}
	\frac{d}{dt} \bigg\{ \frac{d}{2} \io |\na v|^2+\frac{1}{2}\io v^2 -\io uv +D\io u \ln u\bigg\}
	=-\io v_t^2 -\io \Big| D\frac{\na u}{\sqrt{u}}-\sqrt{u}\na v\Big|^2,
\ee
valid along suitably regular trajectories (\cite{NSY}).
Secondly, the attractant concentration $v$ satisfies an inhomogeneous linear parabolic equation, accessible to
classical analysis based on smoothing properties of corresponding heat semigroups.
Suitable combination of these fundamental features has substantially influenced previous studies on (\ref{00}),
and has thereby, inter alia, facilitated the discovery of dichotomies between globally smooth behavior
on the one hand (\cite{NSY}, \cite{cao_small}), and the occurrence of singularity formation on the other (\cite{horstmann_wang},
\cite{win_JMPA}, \cite{win_NON}, \cite{mizoguchi_win});
In addition to these and partially even farther-reaching results obtained for certain parabolic-elliptic simplifications
(\cite{biler1998}, \cite{nagai1995}, \cite{nagai2001}, \cite{perthame_ejde2006}, \cite{perthame_crmas2004}, \cite{senba_suzuki2001},
\cite{bai_zhou_MAAN}), taxis-driven blow-up has also beed detected in closely related complex models
(cf.~\cite{cao_fuest}, \cite{chiyoda_ACAP2020}, \cite{fujie_jiang}, \cite{fujie_senba}, \cite{jin_wang}, \cite{youshan_zhian_M3AS},
\cite{taowin_JEMS} and \cite{xinyu_tu_M3AS} for a small selection).\abs
{\bf A Keller-Segel-type model with anisotropic production of signals.} \quad
In the study of clustering and pattern formation among autophoretic colloids, the authors in \cite{cates_prl2015}
found that the chemical is produced by the colloid asymmetrically due to the anisotropic
properties of Janus particles, and they introduced an additional term $\na \cdot (u\na v)$
to describe a certain anisotropic correction to the isotropic signal production term $+u$.
As a consequence, the signal evolution is accordingly characterized by an equation of the form
\be{002}
	v_t =  d \Del v +\na \cdot (u\na v)- v + u.
\ee
Although a number of variants of (\ref{00}) that involve various alternative types of migration mechanisms
such as nonlinear diffusion or modified chemotactic responses have been extensively studied
in the literature (\cite{BBTW}), possible effects of such anisotropies in chemoattractant production
have been much less explored so far.
In fact, when embedded into the corresponding initial-boundary value problem
\be{0}
	\left\{ \begin{array}{ll}	
	u_t = D \Del u - \na \cdot (u\na v),
	\qquad & x\in\Om, \ t>0, \\[1mm]
	v_t =  d \Del v + \na \cdot (u\na v) - v + u,
	\qquad & x\in\Om, \ t>0, \\[1mm]
	\frac{\pa u}{\pa\nu}=\frac{\pa v}{\pa\nu}=0,
	\qquad & x\in\pO, \ t>0, \\[1mm]
	u(x,0)=u_0(x), \quad v(x,0)=v_0(x),
	\qquad & x\in\Om,
	\end{array} \right.
\ee
to be subsequently considered in a smoothly bounded domain $\Om\subset\R^n$,
this anisotropic correction term $\na \cdot (u\na v)$ in the second equation brings about two analytical obstacles:
It does not only destroy the Lypunov functional structure (\ref{001}) for the original Keller-Segel model;
beyond this, the principal part in the second subsystem of (\ref{0}) thereby loses its linear structure and even contains
a diffusion degeneracy that potentially might counteract higher-order regularity properties.\abs
In line with this, the corresponding analytical literature so far seems limited to the recent study \cite{taowin_JMPA}
in which (\ref{0}) is examined for $n\le 5$ and under the restriction that the difference $|D-d|$ of the linear parts
in both diffusion mechanisms is suitably small.
Within this framework, by means of a non-symmetrically coupled gradient estimate technique
along with a self-mapping argument and a refined H\"older regularity analysis
a result on global existence of bounded classical solutions is derived in \cite{taowin_JMPA}.
Although this markedly distinguishes (\ref{0}) from the classical Keller-Segel model (\ref{00}) with its well-known
core property to generate blow-up when $n\ge 2$,
it leaves open the question how far the introduction of anisotropic corrections in signal production prevents
chemotactic collapse also in the presence of strongly different diffusion rates;
in fact, this question seems of particular relevance in cases in which, as typically seen in nature, individuals
in the considered population move at velocities significantly smaller than signal molecules do.\abs
{\bf Main results.} \quad
The focus of the present work will accordingly be set on the development of a basic solution theory for (\ref{0})
in settings of arbitrary $D>0$ and $d>0$, where mainly for technical purposes we shall concentrate on the case when $n\le 5$.
In order to circumvent obstacles linked to the diffusion degeneracy in the second equation from (\ref{0}), our analysis in this regard
will be based on a variational approach concentration on the evolution of spatial integrals that exclusively
involve zero-order expressions.
In order to nevertheless achieve suitably far-reaching information, the core part of our considerations will trace
functionals of the form
\be{003}
	\io v^2 e^{(w+1)^\al} + b \io (w+1) e^{(w+1)^\al},
	\qquad w:=u+v,
	\qquad b>0,
\ee
along trajectories (see Lemma \ref{lem34}, Lemma \ref{lem36} and Lemma \ref{lem37}).\abs
A priori estimates accordingly implied for solutions to certain regularized variants of (\ref{0}) (see (\ref{0eps}))
will not only lead to a statement on global existence within
a fairly natural notion of weak solvability, but furthermore provide time-independent bounds for $u$ in an
Orlicz class smaller than $L^p(\Om)$ for each finite $p$;
in particular, our following main result rules out any $L^p$ blow-up phenomenon in (\ref{0}) both in finite or in infinite time,
contrary to the situation in the
multi-dimensional version of (\ref{00}) in which finite-time explosions actually occur in each of the spaces $L^p(\Om)$
with $p>\frac{n}{2}$ when $n\ge 2$ (\cite{herrero_velazquez}, \cite{win_JMPA}, \cite[Lemma 3.2]{BBTW}):
\begin{theo}\label{theo14}
  Let $n\le 5$ and $\Om\subset\R^n$ be a bounded domain with smooth boundary, let $D>0$ and $d>0$ be arbitrary,
  and let $K>0$.
  Then there exist $\al=\al(K)>0$ and $C(K)>0$ with the property that whenever
  \be{init}
	u_0\in W^{1,\infty}(\Om)
	\quad \mbox{and} \quad
	v_0\in W^{1,\infty}(\Om) \quad \mbox{are nonnegative}
  \ee
  and such that
  \be{K}
	\|u_0\|_{L^\infty(\Om)}
	+ \|v_0\|_{L^\infty(\Om)}
	\le K,
  \ee
  one can find nonnegative functions
  \be{reg}
	\lbal
	u\in L^2_{loc}([0,\infty);W^{1,2}(\Om))
	\qquad \mbox{and} \\[1mm]
	v\in L^2_{loc}([0,\infty);W^{1,2}(\Om))
	\ear
  \ee
  such that
  \be{14.1}
	\io e^{u^\al(\cdot,t)} \le C(K)
	\qquad \mbox{for a.e.~} t>0
  \ee
  and
  \be{14.2}
	\|v(\cdot,t)\|_{L^\infty(\Om)} \le C(K)
	\qquad \mbox{for a.e.~} t>0,
  \ee
  and that $(u,v)$ forms a global weak solution of (\ref{0}) in the sense that
  \be{wu}
	- \int_0^\infty \io u\vp_t - \io u_0\vp(\cdot,0)
	= - D \int_0^\infty \io \na u\cdot\na\vp
	+ \int_0^\infty \io u\na v\cdot\na\vp
  \ee
  and
  \be{wv}
	- \int_0^\infty \io v\vp_t - \io v_0\vp(\cdot,0)
	= - d \int_0^\infty \io\na v\cdot\na\vp
	- \int_0^\infty \io u\na v\cdot\na\vp
	- \int_0^\infty \io v\vp
	+ \int_0^\infty \io u\vp
  \ee
  hold for each $\vp\in C_0^\infty(\bom\times [0,\infty))$.
\end{theo}
\mysection{Preliminaries}
\subsection{Two families of interpolation inequalities}
A key role in our analysis will be played by two functional inequalities which can be viewed as far relatives of Ehrling's inequality,
and which will be decisive in appropriately estimating zero-order expressions related to the source term $+u$ appearing
in the second equation in (\ref{0}).
In view of our ambition to subsequently consider solutions to the approximate variants of (\ref{0}) introduced in (\ref{0eps})
below, these inequalities will need to suitably cope with the appearance of a regularization parameter $\eps$ therein, and with
consequences thereof on a reduced strength of the diffusion mechanism determining the evolution of the second solution component.\abs
The first of these inequalities will be used in revealing quasi-energy properties enjoyed by certain combinations of functionals
that exhibit essentially cubic growth with respect to both solution components (see Lemma \ref{lem4}):
\begin{lem}\label{lem1}
  Let $n\ge 1$ and $\Om\subset\R^n$ be a bounded domain with smooth boundary.
  Then for each $\eta>0$ one can find $\Lam_1(\eta)>0$ such that whenever $\vp\in C^1(\bom)$ is nonnegative,
  \be{1.1}
	\io \frac{(\vp+1)^3}{1+\eps\vp}
	\le \eta \io \frac{\vp+1}{1+\eps\vp} |\na\vp|^2
	+ \Lam_1(\eta) \cdot \bigg\{ \io (\vp+1) \bigg\}^3
	\qquad \mbox{for all } \eps\in (0,1).
  \ee
\end{lem}
\proof
  An interpolation relying on the compactness of the embedding $W^{1,2}(\Om) \hra L^2(\Om)$ yields $c_1=c_1(\eta)>0$ such that
  \be{1.2}
	\|\psi\|_{L^2(\Om)}^2 \le \frac{4\eta}{9} \|\na\psi\|_{L^2(\Om)}^2
	+ c_1 \|\psi\|_{L^\frac{2}{3}(\Om)}^2
	\qquad \mbox{for all } \psi\in C^1(\bom).
  \ee
  Noting that for
  \bas
	\rho_\eps(\xi):=\frac{\sqrt{\xi+1}^3}{\sqrt{1+\eps\xi}},
	\qquad \xi\ge 0, \ \eps\in (0,1),
  \eas
  we have
  \bas
	\rho_\eps'(\xi)
	&=& \frac{3}{2} \cdot \frac{\sqrt{\xi+1}}{\sqrt{1+\eps\xi}}
	- \frac{\eps}{2} \cdot \frac{\sqrt{\xi+1}^3}{\sqrt{1+\eps\xi}^3} \\
	&=& \frac{1}{2} \cdot \sqrt{\frac{\xi+1}{1+\eps\xi}} \cdot \frac{3-\eps+2\eps\xi}{1+\eps\xi}
	\qquad \mbox{for all $\xi\ge 0$ and } \eps\in (0,1)
  \eas
  and hence
  \bas
	0\le \rho_\eps'(\xi)
	\le \frac{1}{2} \cdot \sqrt{\frac{\xi+1}{1+\eps\xi}} \cdot \frac{3+3\eps\xi}{1+\eps\xi}
	= \frac{3}{2} \cdot \sqrt{\frac{\xi+1}{1+\eps\xi}}
	\qquad \mbox{for all $\xi\ge 0$ and } \eps\in (0,1),
  \eas
  for fixed nonnegative $\vp\in C^1(\bom)$ we obtain from (\ref{1.2}) that
  \bas
	\io \frac{(\vp+1)^3}{1+\eps\vp}
	&=& \big\| \rho_\eps(\vp)\big\|_{L^2(\Om)}^2 \\
	&\le& \frac{4\eta}{9} \big\| \na\rho_\eps(\vp)\big\|_{L^2(\Om)}^2
	+ c_1 \big\| \rho_\eps(\vp)\big\|_{L^\frac{2}{3}(\Om)}^2 \\
	&=& \frac{4\eta}{9} \cdot \io \rho_\eps'^2(\vp) |\na\vp|^2
	+ c_1 \cdot \bigg\{ \io \rho_\eps^\frac{2}{3}(\vp)\bigg\}^3 \\
	&\le& \eta \io \frac{\vp+1}{1+\eps\vp} |\na\vp|^2
	+ c_1 \cdot \bigg\{ \io \rho_\eps^\frac{2}{3}(\vp)\bigg\}^3
	\qquad \mbox{for all } \eps\in (0,1).
  \eas
  Since
  \bas
	c_1 \cdot \bigg\{ \io \rho_\eps^\frac{2}{3}(\vp)\bigg\}^3
	= c_1 \cdot \bigg\{ \io \frac{\vp+1}{(1+\eps\vp)^\frac{1}{3}} \bigg\}^3
	\le c_1 \cdot \bigg\{ \io (\vp+1) \bigg\}^3
	\qquad \mbox{for all } \eps\in (0,1)
  \eas
  by nonnegativity of $\eps\vp$ for any such $\eps$, this yields (\ref{1.1}) with $\Lam_1(\eta):=c_1$.
\qed
Establishing a second and more subtle relation will require the following statement on zero-order interpolation as a preliminary.
\begin{lem}\label{lem2}
  If $n\ge 1$ and $\Om\subset\R^n$ is a bounded domain with smooth boundary,
  and if $\al\in (0,1)$, $\mu>0$ and $\eta>0$, then there exists $\Lam_2(\eta,\al,\mu)>0$ with the property that
  any nonnegative $\vp\in C^0(\bom)$ fulfilling
  \be{2.1}
	\io \vp \le \mu
  \ee
  satisfies
  \be{2.2}
	\bigg\{ \io \frac{(\vp+1)^\frac{\al+1}{2}}{(1+\eps\vp)^\frac{1}{2}} \cdot e^{\frac{1}{2} (\vp+1)^\al} \bigg\}^2
	\le \eta \io \frac{(\vp+1)^{\al+1}}{1+\eps\vp} \cdot e^{(\vp+1)^\al}
	+ \Lam_2(\eta,\al,\mu)
	\qquad \mbox{for all } \eps\in (0,1).
  \ee
\end{lem}
\proof
  We let $N=N(\al)\ge 1$ be such that
  \bas
	\al N^\al \ge 1-\al,
  \eas
  and note that then for
  \be{2.22}
	\rho(\xi):=(\xi+N)^\frac{\al-1}{2} e^{\frac{1}{2}(\xi+1)^\al},
	\qquad \xi\ge 0,
  \ee
  we have
  \bea{2.23}
	\rho'(\xi)
	&=& \frac{\al}{2} (\xi+N)^\frac{\al-1}{2} (\xi+1)^{\al-1} e^{\frac{1}{2}(\xi+1)^\al}
	- \frac{1-\al}{2} (\xi+N)^\frac{\al-3}{2} e^{\frac{1}{2}(\xi+1)^\al} \nn\\
	&=& \frac{1}{2} (\xi+N)^\frac{\al-3}{2} e^{\frac{1}{2}(\xi+1)^\al} 	
		\cdot \big\{ \al(\xi+N)(\xi+1)^{\al-1} - (1-\al) \big\} \nn\\
	&=& \frac{1}{2} (\xi+N)^\frac{\al-3}{2} e^{\frac{1}{2}(\xi+1)^\al} 	
		\cdot \Big\{ \al(\xi+N)^\al \cdot \Big(\frac{\xi+N}{\xi+1}\Big)^{1-\al} - (1-\al)\Big\} \nn\\
	&\ge& \frac{1}{2} (\xi+N)^\frac{\al-3}{2} e^{\frac{1}{2}(\xi+1)^\al} 	
		\cdot \big\{ \al N^\al - (1-\al)\big\}
	\ge 0
	\qquad \mbox{for all } \xi\ge 0,
  \eea
  because $\al<1$.
  For $0\le \vp \in C^0(\bom)$ fulfilling (\ref{2.1}), and for arbitrary $a>0$, splitting
  \be{2.3}
	\io \frac{(\vp+1)^\frac{\al+1}{2}}{(1+\eps\vp)^\frac{1}{2}} \cdot e^{\frac{1}{2} (\vp+1)^\al}
	= \int_{\{\vp<a\}} \frac{(\vp+1)^\frac{\al+1}{2}}{(1+\eps\vp)^\frac{1}{2}} \cdot e^{\frac{1}{2} (\vp+1)^\al}
	+ \int_{\{\vp\ge a\}} \frac{(\vp+1)^\frac{\al+1}{2}}{(1+\eps\vp)^\frac{1}{2}} \cdot e^{\frac{1}{2} (\vp+1)^\al}
	\qquad \eps\in (0,1),
  \ee
  we can therefore estimate
  \bea{2.4}
	\int_{\{\vp<a\}} \frac{(\vp+1)^\frac{\al+1}{2}}{(1+\eps\vp)^\frac{1}{2}} \cdot e^{\frac{1}{2} (\vp+1)^\al}
	&\le& \int_{\{\vp<a\}} (\vp+N)^\frac{\al+1}{2} e^{\frac{1}{2}(\vp+1)^\al} \nn\\
	&=& \int_{\{\vp<a\}} (\vp+N) \rho(\vp) \nn\\
	&\le& \rho(a) \int_{\{\vp<a\}} (\vp+N) \nn\\
	&\le& (\mu+N|\Om|) \rho(a)
	\qquad \mbox{for all } \eps\in (0,1)
  \eea
  according to (\ref{2.1}).
  Apart from that, simply using that
  \bas
	\frac{(1+\eps\vp)^\frac{1}{2}}{(\vp+1)^\frac{\al+1}{2}}
	\le \frac{(1+\eps\vp)^\frac{1}{2}}{(\vp+1)^\frac{1}{2}}
	\le 1
  \eas
  for $\eps\in (0,1)$, we see on writing
  $I_\eps(\vp):=\io \frac{(\vp+1)^{\al+1}}{1+\eps\vp} \cdot e^{(\vp+1)^\al}$, $\eps\in (0,1)$, that
  \bas
	\int_{\{\vp\ge a\}} \frac{(\vp+1)^\frac{\al+1}{2}}{(1+\eps\vp)^\frac{1}{2}} \cdot e^{\frac{1}{2} (\vp+1)^\al}
	&=& \int_{\{\vp\ge a\}} \Big\{ \frac{(\vp+1)^{\al+1}}{1+\eps\vp} \cdot e^{(\vp+1)^\al} \Big\}
	\cdot \frac{(1+\eps\vp)^\frac{1}{2}}{(\vp+1)^\frac{\al+1}{2}} \cdot e^{-\frac{1}{2}(\vp+1)^\al} \\
	&\le& \int_{\{\vp\ge a\}} \Big\{ \frac{(\vp+1)^{\al+1}}{1+\eps\vp} \cdot e^{(\vp+1)^\al} \Big\}
	\cdot e^{-\frac{1}{2}(\vp+1)^\al} \\[1mm]
	&\le& e^{-\frac{1}{2}(a+1)^\al} I_\eps(\vp)
	\qquad \mbox{for all } \eps\in (0,1),
  \eas
  whence by (\ref{2.3}) and (\ref{2.4}),
  \be{2.5}
	\io \frac{(\vp+1)^\frac{\al+1}{2}}{(1+\eps\vp)^\frac{1}{2}} \cdot e^{\frac{1}{2} (\vp+1)^\al}
	\le c_1 \rho(a) + e^{-\frac{1}{2}(a+1)^\al} I_\eps(\vp)
	\qquad \mbox{for all $\eps\in (0,1)$ and } a>0
  \ee
  with $c_1\equiv c_1(\al,\mu):=\mu+N|\Om|$.\abs
  We now fix $\eta>0$ and let
  \be{2.6}
	a_\eps \equiv a_\eps(\eta,\vp):=
	\bigg\{ 2 \ln_+ \sqrt{\frac{4 I_\eps(\vp)}{\eta}} \bigg\}^\frac{1}{\al},
	\qquad
	\eps\in (0,1),
  \ee
  where $\ln_+ \xi:=\max\{0 \, , \, \ln \xi\}$ for $\xi>0$.
  Then in the case when $\eps\in (0,1)$ is such that
  \be{2.66}
	a_\eps \le a_0\equiv a_0(\eta,\al,\mu):=
	\Big(\frac{4c_1\sqrt{e}}{\eta}\Big)^\frac{2}{1-\al},
  \ee
  we evidently have $\ln \frac{4 I_\eps(\vp)}{\eta} < a_0^\al$,
  that is,
  \bas
	I_\eps(\vp) \le \frac{\eta}{4} e^{a_0^\al},
  \eas
  so that (\ref{2.5}) together with (\ref{2.23}) guarantees that
  \be{2.7}
	\io \frac{(\vp+1)^\frac{\al+1}{2}}{(1+\eps\vp)^\frac{1}{2}} \cdot e^{\frac{1}{2} (\vp+1)^\al}
	\le c_2\equiv c_2(\eta,\al,\mu):= c_1 \rho(a_0)
	+ e^{-\frac{1}{2}(a_0+1)^\al} \cdot \frac{\eta}{4} e^{a_0^\al}.
  \ee
  If, conversely,
  \be{2.8}
	a_\eps>a_0,
  \ee
  then by (\ref{2.6}),
  \be{2.81}
	e^{-\frac{1}{2}(a_\eps+1)^\al} I_\eps(\vp)
	\le e^{-\frac{1}{2}a_\eps^\al} I_\eps(\vp)
	= e^{-\ln \sqrt{\frac{4 I_\eps(\vp)}{\eta}}} I_\eps(\vp)
	= \frac{\sqrt{\eta}}{2} \cdot \sqrt{I_\eps(\vp)},
  \ee
  while according to our definition of $a_0$ in (\ref{2.66}), and again since $\al< 1$,
  we may estimate $(a_\eps+1)^\al \le a_\eps^\al + 1$ to see that
  \bea{2.82}
	c_1 \rho(a_\eps)
	&=& c_1 (a_\eps+N)^\frac{\al-1}{2} e^{\frac{1}{2}(a_\eps+1)^\al} \nn\\
	&\le& c_1 \sqrt{e} a_0^\frac{\al-1}{2} e^{\frac{1}{2} a_\eps^\al} \nn\\
	&=& c_1\sqrt{e} a_0^\frac{\al-1}{2} \cdot \sqrt{\frac{4I_\eps(\vp)}{\eta}} \nn\\
	&=& \frac{\sqrt{\eta}}{2} \cdot \sqrt{I_\eps(\vp)}.
  \eea
  In view of (\ref{2.81}) and (\ref{2.82}), an application of (\ref{2.5}) shows that whenever (\ref{2.8}) holds, we have
  \bas
	\bigg\{ \io\frac{(\vp+1)^\frac{\al+1}{2}}{(1+\eps\vp)^\frac{1}{2}} \cdot e^{\frac{1}{2} (\vp+1)^\al} \bigg\}^2
	\le \Big\{ \frac{\sqrt{\eta}}{2} \cdot \sqrt{I_\eps(\vp)} + \frac{\sqrt{\eta}}{2} \cdot \sqrt{I_\eps(\vp)} \Big\}^2
	= \eta I_\eps(\vp),
  \eas
  which in conjunction with (\ref{2.7}) shows that (\ref{2.2}) is valid for any choice of $\eps\in (0,1)$ if we let
  $\Lam_2(\eta,\al,\mu):=c_2^2$.
\qed
A second preparation consists in the following elementary observation.
\begin{lem}\label{lem31}
  Let $\al>0$, and for $\eps\in (0,1)$ let
  \be{31.1}
	\rho_\eps(\xi):=\frac{(\xi+1)^\frac{\al+1}{2}}{(1+\eps\xi)^\frac{1}{2}} \cdot e^{\frac{1}{2}(\xi+1)^\al},
	\qquad \xi\ge 0.
  \ee
  Then
  \be{31.2}
	0 \le \rho_\eps'(\xi)
	\le \al \cdot \frac{(\xi+1)^\frac{3\al-1}{2}}{(1+\eps\xi)^\frac{1}{2}} \cdot e^{\frac{1}{2}(\xi+1)^\al}
	+ \frac{(\al+1)^\frac{3}{2}}{\sqrt{\al}} \cdot e^\frac{\al+1}{2\al}
	\qquad \mbox{for all } \xi\ge 0.
  \ee
\end{lem}
\proof
  We compute
  \be{31.3}
	\rho_\eps'(\xi)
	= \frac{\al}{2} \cdot \frac{(\xi+1)^\frac{3\al-1}{2}}{(1+\eps\xi)^\frac{1}{2}} \cdot e^{\frac{1}{2}(\xi+1)^\al}
	+ \frac{\al+1}{2} \cdot \frac{(\xi+1)^\frac{\al-1}{2}}{(1+\eps\xi)^\frac{1}{2}} \cdot e^{\frac{1}{2}(\xi+1)^\al}
	- \frac{\eps}{2} \cdot \frac{(\xi+1)^\frac{\al+1}{2}}{(1+\eps\xi)^\frac{3}{2}} \cdot e^{\frac{1}{2}(\xi+1)^\al},
	\qquad \xi\ge 0,
  \ee
  and thus obtain on dropping the nonnegative first summand here that
  \bas
	\rho_\eps'(\xi)
	&\ge& \frac{(\xi+1)^\frac{\al-1}{2}}{(1+\eps\xi)^\frac{1}{2}} \cdot e^{\frac{1}{2}(\xi+1)^\al} \cdot \Big\{
	\frac{\al+1}{2} - \frac{\eps}{2} \cdot \frac{\xi+1}{1+\eps\xi} \Big\} \\
	&\ge& \frac{(\xi+1)^\frac{\al-1}{2}}{(1+\eps\xi)^\frac{1}{2}} \cdot e^{\frac{1}{2}(\xi+1)^\al} \cdot \Big\{
	\frac{\al+1}{2} - \frac{\eps}{2} \cdot \frac{\xi+\frac{1}{\eps}}{1+\eps\xi} \Big\} \\
	&=& \frac{(\xi+1)^\frac{\al-1}{2}}{(1+\eps\xi)^\frac{1}{2}} \cdot e^{\frac{1}{2}(\xi+1)^\al} \cdot \frac{\al}{2}
	\qquad \mbox{for all } \xi\ge 0,
  \eas
  which particularly yields the left inequality in (\ref{31.2}).
  To verify the right one, we first observe that if $\xi\ge 0$ is such that
  $(\xi+1)^\al \ge \frac{\al+1}{\al}$, then trivially estimating the rightmost summand in (\ref{31.3}) we see that
  \bas
	\rho_\eps'(\xi)
	&\le& \frac{(\xi+1)^\frac{3\al-1}{2}}{(1+\eps\xi)^\frac{1}{2}} \cdot e^{\frac{1}{2}(\xi+1)^\al} \cdot \Big\{
	\frac{\al}{2} + \frac{\al+1}{2} \cdot \frac{1}{(\xi+1)^\al} \Big\} \\
	&\le& \frac{(\xi+1)^\frac{3\al-1}{2}}{(1+\eps\xi)^\frac{1}{2}} \cdot e^{\frac{1}{2}(\xi+1)^\al} \cdot \Big\{
	\frac{\al}{2} + \frac{\al+1}{2} \cdot \frac{\al}{\al+1} \Big\} \\
	&=& \frac{(\xi+1)^\frac{3\al-1}{2}}{(1+\eps\xi)^\frac{1}{2}} \cdot e^{\frac{1}{2}(\xi+1)^\al} \cdot \al.
  \eas
  Since, on the other hand, for any $\xi\ge 0$ fulfilling
  $(\xi+1)^\al < \frac{\al+1}{\al}$ we have
  \bas
	\rho_\eps'(\xi)
	&\le& \frac{\al}{2} \cdot (\xi+1)^\frac{3\al}{2} \cdot e^{\frac{1}{2}(\xi+1)^\al}
	+ \frac{\al+1}{2} \cdot (\xi+1)^\frac{\al}{2} \cdot e^{\frac{1}{2}(\xi+1)^\al} \\
	&\le& \frac{\al}{2} \cdot \Big(\frac{\al+1}{\al}\Big)^\frac{3}{2} \cdot e^\frac{\al+1}{2\al}
	+ \frac{\al+1}{2} \cdot\Big(\frac{\al+1}{\al}\Big)^\frac{1}{2} \cdot e^\frac{\al+1}{2\al} \\
	&=& \Big(\frac{\al+1}{\al}\Big)^\frac{1}{2} \cdot \Big\{ \frac{\al}{2} \cdot \frac{\al+1}{\al} + \frac{\al+1}{2}\Big\}
		\cdot e^\frac{\al+1}{2\al},
  \eas
  rearranging shows that (\ref{31.2}) holds in both these cases.
\qed
Indeed, we can thereby derive a family of relatives of (\ref{1.1}) which contain some superalgebraically growing quantities,
and which will thereby form a crucial ingredient to our analysis related to the bounds claimed in (\ref{14.1}), as to be detailed
in Lemma \ref{lem37}.
\begin{lem}\label{lem32}
  Let $n\ge 1$ and $\Om\subset\R^n$ be a bounded domain with smooth boundary,
  and let $\mu>0$.
  Then there exists $\Lam_3(\mu)>0$ with the property that whenever $\al\in (0,\min\{1,\frac{2}{n} \})$,
  one can find $\Lam_4(\al,\mu)>0$ such that if $\vp\in C^1(\bom)$ is nonnegative and such that (\ref{2.1}) holds, then
  \bea{32.1}
	\io \frac{(\vp+1)^{\al+1}}{1+\eps\vp} \cdot e^{(\vp+1)^\al}
	&\le& \Lam_3(\mu) \al^2 \io \frac{(\vp+1)^{2\al-1}}{1+\eps\vp} \cdot e^{(\vp+1)^\al} |\na\vp|^2 \nn\\
	& & + \Lam_4(\al,\mu) \io |\na\vp|^2
	+ \Lam_4(\al,\mu)
	\qquad \mbox{for all } \eps\in (0,1).
  \eea
\end{lem}
\proof
  We let $q:=\max\{\frac{2n}{n+2},1\} \in [1,2)$, and may then rely on the continuity of the embedding $W^{1,q}(\Om) \hra L^2(\Om)$
  to find $c_1>0$ such that
  \be{32.2}
	\|\psi\|_{L^2(\Om)}^2 \le c_1 \|\na\psi\|_{L^q(\Om)}^2
	+ c_1 \|\psi\|_{L^1(\Om)}^2
	\qquad \mbox{for all } \psi\in C^1(\bom).
  \ee
  Fixing $\al\in (0,\min\{1,\frac{2}{n} \}]$, $\mu>0$ and $0\le\vp\in C^1(\bom)$ such that $\io \vp\le \mu$, we thus obtain that
  if we let $(\rho_\eps)_{\eps\in (0,1)}$ be as in Lemma \ref{lem31}, then
  \bea{32.3}
	\io \frac{(\vp+1)^{\al+1}}{1+\eps\vp} \cdot e^{(\vp+1)^\al}
	= \|\rho_\eps(\vp)\|_{L^2(\Om)}^2
	\le c_1\|\na\rho_\eps(\vp)\|_{L^q(\Om)}^2
	+ c_1 \|\rho_\eps(\vp)\|_{L^1(\Om)}^2
  \eea
  for all $\eps\in (0,1)$.
  Here, abbreviating $c_2\equiv c_2(\al):=\frac{(\al+1)^\frac{3}{2}}{\sqrt{\al}} \cdot e^\frac{\al+1}{2\al}$ we see that
  thanks to (\ref{31.2}) and the H\"older inequality,
  \bea{32.4}
	c_1 \|\na\rho_\eps(\vp)\|_{L^q(\Om)}^2
	&=& c_1 \|\rho_\eps'(\vp)\na\vp\|_{L^q(\Om)}^2 \nn\\
	&\le& c_1 \cdot \bigg\|
		\Big\{ \al \cdot \frac{(\vp+1)^\frac{3\al-1}{2}}{(1+\eps\vp)^\frac{1}{2}} \cdot e^{\frac{1}{2}(\vp+1)^\al} + c_2
		\big\} \cdot |\na\vp| \bigg\|_{L^q(\Om)}^2 \nn\\
	&\le& 2c_1 \al^2 \cdot \bigg\| \frac{(\vp+1)^\frac{3\al-1}{2}}{(1+\eps\vp)^\frac{1}{2}} \cdot e^{\frac{1}{2}(\vp+1)^\al}
		|\na\vp| \bigg\|_{L^q(\Om)}^2
	+ 2c_1 c_2^2 \|\na\vp\|_{L^q(\Om)}^2 \nn\\
	&=& 2c_1 \al^2 \cdot \bigg\{ \io \Big\{ \frac{(\vp+1)^{2\al-1}}{1+\eps\vp} \cdot e^{(\vp+1)^\al} |\na\vp|^2 \Big\}^\frac{q}{2}
		\cdot (\vp+1)^\frac{q\al}{2} \bigg\}^\frac{2}{q}
	+ 2c_1 c_2^2 \cdot \bigg\{ \io |\na\vp|^q \bigg\}^\frac{2}{q} \nn\\
	&\le& 2c_1 \al^2 \cdot \bigg\{ \io \frac{(\vp+1)^{2\al-1}}{1+\eps\vp} \cdot e^{(\vp+1)^\al} |\na\vp|^2 \bigg\}
		\cdot \bigg\{ \io (\vp+1)^\frac{q\al}{2-q}\bigg\}^\frac{2-q}{q} \nn\\
	& & + 2c_1 c_2^2 |\Om|^\frac{2-q}{q} \io |\na\vp|^2
	\qquad \mbox{for all } \eps\in (0,1),
  \eea
  where we note that when $n\ge 2$, we have $q=\frac{2n}{n+2}$ and thus $\frac{q\al}{2-q}=\frac{n\al}{2}\le 1$ due to our restriction
  that $\al\le \frac{2}{n}$, while if $n=1$, then $q=1$ and hence $\frac{q\al}{2-q}=\al\le 1$.
  Therefore, regardless of the size of $n$ we may draw on (\ref{2.1}) to estimate
  \bas
	\io (\vp+1)^\frac{q\al}{2-q} \le \io (\vp+1) \le \mu+|\Om|,
  \eas
  so that (\ref{32.4}) ensures that
  \bea{32.5}
	c_1 \|\na\rho_\eps(\vp)\|_{L^q(\Om)}^2
	&\le& 2c_1 \al^2 \cdot (\mu+|\Om|)^\frac{2-q}{q} \io \frac{(\vp+1)^{2\al-1}}{1+\eps\vp} \cdot e^{(\vp+1)^\al} |\na\vp|^2 \nn\\
	& & + 2c_1 c_2^2 |\Om|^\frac{2-q}{q} \io |\na\vp|^2
	\qquad \mbox{for all } \eps\in (0,1).
  \eea
  Since, apart from that, an application of Lemma \ref{lem2} shows that if we let
  $c_3\equiv c_3(\al,\mu):=\Lam_2(\frac{1}{2c_1},\al,\mu)$ with $\Lam_2(\cdot,\cdot,\cdot)$ as found there, then
  \bas
	c_1 \|\rho_\eps(\vp)\|_{L^1(\Om)}^2
	&=& c_1 \cdot \bigg\{ \io \frac{(\vp+1)^\frac{\al+1}{2}}{(1+\eps\vp)^\frac{1}{2}} \cdot e^{\frac{1}{2}(\vp+1)^\al} \bigg\}^2 \\
	&\le& \frac{1}{2} \io \frac{(\vp+1)^{\al+1}}{1+\eps\vp} \cdot e^{(\vp+1)^\al}
	+ c_3
	\qquad \mbox{for all } \eps\in (0,1),
  \eas
  a combination of (\ref{32.3}) with (\ref{32.5}) shows that (\ref{32.1}) holds if we let
  $\Lam_3 \equiv \Lam_3(\mu):=4c_1 \cdot (\mu+|\Om|)^\frac{2-q}{q}$ and
  $\Lam_4\equiv \Lam_4(\al,\mu):=\max\big\{ 4c_1 c_2^2 |\Om|^\frac{2-q}{q} \, , \, 2c_3\big\}$.
\qed
\subsection{A quantitative outcome of a Moser-type iteration}
The following is a consequence of a slightly more general statement on the outcome of a Moser-type iterative reasoning,
as recorded in \cite[Lemma 2.2]{ding_win_NoDEA}.
\begin{lem}\label{lem_moser}
  Let $n\ge 1$ and $\Omega\subset\R^n$ be a bounded domain with smooth boundary, and let $q\in [1,\infty]$ be such that
  $q>\frac{n}{2}$.
  Then for all $L>0$ one can find $\Lam_4(q,L)>0$ with the property that whenever $T\in (0,\infty]$,
  $a\in C^1(\bom\times (0,T))$,
  $f\in C^0(\bom\times (0,T))$ and
  $z\in C^0(\bom\times [0,T)) \cap C^{2,1}(\bom\times (0,T))$ are such that
  \be{a}
	a(x,t)\ge \frac{1}{L}
	\qquad \mbox{for all } (x,t)\in \Om\times (0,T),
  \ee
  \be{f}
	\|f(\cdot,t)\|_{L^q(\Om)} \le L
	\qquad \mbox{for all } t\in (0,T),
  \ee
  and that $z$ is nonnegative with
  \be{m0}
	\left\{ \begin{array}{ll}
	z_t \le \na \cdot \big( a(x,t)\na z\big) + f(x,t) z
	\qquad & x\in\Om, \ t\in (0,T), \\[1mm]
	\frac{\pa z}{\pa\nu} \le 0,
	\qquad & x\in\pO, \ t\in (0,T),
	\end{array} \right.
  \ee
  we have
  \be{m1}
	\|z(\cdot,t)\|_{L^\infty(\Om)} \le \Lam_4(q,L) \cdot
	\max \bigg\{ \|z(\cdot,0)\|_{L^\infty(\Om)} \, , \, \sup_{s\in (0,T)} \|z(\cdot,s)\|_{L^1(\Om)} \bigg\}
	\qquad \mbox{for all } t\in (0,T).
  \ee
\end{lem}
As substantiated in Lemma \ref{lem4} below, our subsequent analysis will make use of Lemma \ref{lem_moser} through the following
consequence thereof.
\begin{cor}\label{cor_moser}
  Suppose that $n\ge 1$ and $\Omega\subset\R^n$ be a bounded domain with smooth boundary, that $q\in [1,\infty]$ is
  such that $q>\frac{n}{2}$, and that $L>0$, and let $\Lam_4(q,L)$ be as in Lemma \ref{lem_moser}.
  Then whenever $T\in (0,\infty]$, $a\in C^1(\bom\times (0,T))$,
  $f\in C^0(\bom\times (0,T))$ and
  $z\in C^0(\bom\times [0,T)) \cap C^{2,1}(\bom\times (0,T))$ are such that $f\ge 0$ and $z\ge 0$ in $\Om\times (0,T)$,
  that (\ref{a}) and (\ref{f}) hold, and that
  \be{m2}
	\left\{ \begin{array}{ll}
	z_t \le \na \cdot \big( a(x,t)\na z\big) + f(x,t)
	\qquad & x\in\Om, \ t\in (0,T), \\[1mm]
	\frac{\pa z}{\pa\nu} \le 0,
	\qquad & x\in\pO, \ t\in (0,T),
	\end{array} \right.
  \ee
  it follows that
  \be{m3}
	\|z(\cdot,t)\|_{L^\infty(\Om)} \le \Lam_4(q,L) \cdot
	\max \bigg\{ \|z(\cdot,0)+1\|_{L^\infty(\Om)} \, , \, \sup_{s\in (0,T)} \|z(\cdot,s)+1\|_{L^1(\Om)} \bigg\}
	\qquad \mbox{for all } t\in (0,T).
  \ee
\end{cor}
\proof
  Letting $\wh{z}:=z+1$, from (\ref{m2}) we obtain that $\frac{\pa\wh{z}}{\pa\nu}\le 0$ on $\pO\times (0,T)$, and that
  since both $f$ and $z$ are nonnegative,
  \bas
	\wh{z}_t = z_t
	\le \na\cdot \big(a(x,t)\na z\big) + f(x,t)
	= \na\cdot \big(a(x,t)\na\wh{z}\big) + \frac{1}{z+1} \cdot f(x,t)\wh{z}
	\le \na\cdot \big(a(x,t)\na\wh{z}\big) + f(x,t)\wh{z}
  \eas
  in $\Om\times (0,T)$. Relying on (\ref{a}) and (\ref{f}), an application of Lemma \ref{lem_moser} therefore shows that
  \bas
	\|\wh{z}(\cdot,t)\|_{L^\infty(\Om)} \le \Lam_4(q,L) \cdot
	\max \bigg\{ \|\wh{z}(\cdot,0)\|_{L^\infty(\Om)} \, , \, \sup_{s\in (0,T)} \|\wh{z}(\cdot,s)\|_{L^1(\Om)} \bigg\}
	\qquad \mbox{for all } t\in (0,T),
  \eas
  from which (\ref{m3}) already follows due to the fact that $|z|=z\le\wh{z}$ by nonnegativity of $z$.
\qed
\mysection{An approximate variant of (\ref{0}) and basic properties}
As will turn out below, a regularization of (\ref{0}) which does not only admit global classical solutions, but which simultaneously
also is compatible with some favorable structural properties formally enjoyed by (\ref{0}), to be discovered in Lemma \ref{lem11}
and Lemma \ref{lem37}, can be achieved by considering the family of problems given by
\be{0eps}
	\left\{ \begin{array}{ll}	
	u_{\eps t} = D \Del \ueps - \na\cdot\Big( \frac{\ueps}{1+\eps\ueps} \na\veps\Big),
	\qquad & x\in\Om, \ t>0, \\[1mm]
	v_{\eps t} = d\Del\veps + \na\cdot\Big( \frac{\ueps}{1+\eps\ueps} \na\veps\Big) - \veps + \frac{\ueps}{1+\eps\ueps},
	\qquad & x\in\Om, \ t>0, \\[1mm]
	\frac{\pa \ueps}{\pa\nu}=\frac{\pa \veps}{\pa\nu}=0,
	\qquad & x\in\pO, \ t>0, \\[1mm]
	\ueps(x,0)=u_0(x), \quad \veps(x,0)=v_0(x),
	\qquad & x\in\Om,
	\end{array} \right.
\ee
for $\eps\in (0,1)$.
As a first step toward verifying this, let us record the following statement on local existence and extensibility therefor,
and on two basic properties related to the evolution of corresponding mass functionals.
\begin{lem}\label{lem_loc}
  Let $n\ge 1$ and $\Om\subset\R^n$ be a smoothly bounded domain, let $D>0$ and $d>0$, and assume (\ref{init}).
  Then for each $\eps\in (0,1)$, there exist $\tme\in (0,\infty]$ and nonnegative functions
  \bas
	\lbal
	\ueps\in \bigcap_{p>n} C^0([0,\tme);W^{1,p}(\Om)) \cap C^{2,1}(\bom\times (0,\tme))
	\qquad \mbox{and} \\[1mm]
	\veps\in \bigcap_{p>n} C^0([0,\tme);W^{1,p}(\Om)) \cap C^{2,1}(\bom\times (0,\tme))
	\ear
  \eas
  such that $(\ueps,\veps)$ solves (\ref{0eps}) in the classical sense in $\Om\times (0,\tme)$, and that
  \be{ext}
	\mbox{if $\tme<\infty$, \quad then \quad}
	\limsup_{t\nearrow\tme} \Big\{ \|\ueps(\cdot,t)\|_{W^{1,p}(\Om)} + \|\veps(\cdot,t)\|_{W^{1,p}(\Om)} \Big\}=\infty
	\mbox{ for all } p>n.
  \ee
  This solution has the additional property that
  \be{mass}
	\io \ueps(\cdot,t) = \io u_0
	\quad \mbox{and} \quad
	\io \veps(\cdot,t) \le \max \bigg\{ \io u_0 \, , \, \io v_0\bigg\}
	\qquad \mbox{for all } t\in (0,\tme).
  \ee
\end{lem}
\proof
  The statement concerning local existence and the extensibility criterion in (\ref{ext}) directly
  follows from the standard parabolic theory developed in \cite{amann}, while the mass property (\ref{mass})
  readily results from straightforward integration in (\ref{0eps})
  along with a simple ODE comparsion argument.
\qed
From now on, we shall fix $n\le 5$ and a bounded domain $\Om\subset\R^n$ with smooth boundary, as well as numbers $D>0$ and $d>0$
and functions $u_0$ and $v_0$ fulfilling (\ref{init}), and given $\eps\in (0,1)$ we then let $\tme$ and $(\ueps,\veps)$
be as obtained in Lemma \ref{lem_loc}.\abs
Our derivation of a first regularity property beyond those in (\ref{mass}) will make essential use of the circumstance
that the regularization underlying (\ref{0eps}) treats the crucial ingredients $\pm \na\cdot (u\na v)$ to (\ref{0})
in a synchronous manner, thus facilitating a favorable cancellation encountered when adding both parabolic equations in (\ref{0eps}):
\begin{lem}\label{lem_w}
  For $\eps\in (0,1)$, let
  \be{w}
	\weps(x,t):=\ueps(x,t)+\veps(x,t),
	\qquad x\in\bom, \ t\in [0,\tme).
  \ee
  Then $\weps$ lies in $\bigcap_{p>n} C^0([0,\tme);W^{1,p}(\Om)) \cap C^{2,1}(\bom\times (0,\tme))$ and satisfies
  \be{0w}
	\lball
	w_{\eps t} = D\Del\weps + (d-D) \Del\veps - \veps + \frac{\ueps}{1+\eps\ueps},
	\qquad & x\in\Om, \ t\in (0,\tme), \\[1mm]
	\frac{\pa \weps}{\pa\nu}=0,
	\qquad & x\in\pO, \ t\in (0,\tme), \\[1mm]
	\weps(x,0)=u_0(x)+v_0(x),
	\qquad & x\in\Om,
	\end{array} \right.
  \ee
  in the classical sense.
\end{lem}
\proof
  In view of Lemma \ref{lem_loc}, this can be seen by combining the two sub-problems of (\ref{0eps}) in a straightforward manner.
\qed
The plain structure of (\ref{0w}) allows for simple testing procedures, a general template for which is recorded in the following.
\begin{lem}\label{lem33}
  Let $\rho\in C^2([0,\infty))$ be such that $\rho'\ge 0$ and $\rho''\ge 0$, and assume (\ref{init}). Then
  \be{33.1}
	\frac{d}{dt} \io \rho(\weps)
	+ \frac{D}{2} \io \rho''(\weps) |\na\weps|^2
	\le \frac{(d-D)^2}{2D} \io \rho''(\weps) |\na\veps|^2
	+ \io (\weps+1) \rho'(\weps)
  \ee
  for all $t\in (0,\tme)$ and $\eps\in (0,1)$.
\end{lem}
\proof
  Let $\eps\in (0,1)$.
  An integration by parts on the basis of (\ref{0w}) then shows that
  \bea{33.2}
	\frac{d}{dt} \io \rho(\weps)
	&=& \io \rho'(\weps) \na\cdot \big\{ D\na\weps + (d-D)\na\veps\big\}
	+ \io \rho'(\weps) \cdot \Big\{ - \veps + \frac{\ueps}{1+\eps\ueps} \Big\} \nn\\
	&=& - D \io \rho''(\weps) |\na\weps|^2
	- (d-D) \io \rho''(\weps) \na\veps\cdot\na\weps \nn\\
	& & - \io \veps \rho'(\weps)
	+ \io \frac{\ueps}{1+\eps\ueps} \cdot \rho'(\weps)
	\qquad \mbox{for all } t\in (0,\tme).
  \eea
  Here, relying on the nonnegativity of $\rho''$ we may invoke Young's inequality to see that
  for all $t\in (0,\tme)$ we have
  \be{33.3}
	- (d-D) \io \rho''(\weps) \na\veps\cdot\na\weps
	\le \frac{D}{2} \io \rho''(\weps) |\na\weps|^2
	+ \frac{(d-D)^2}{2D} \io \rho''(\weps)|\na\veps|^2,
  \ee
  while using that $\ueps\le\weps$ and that $\rho'\ge 0$ we can estimate
  \be{33.4}
	- \io \veps \rho'(\weps)
	+
	\io \frac{\ueps}{1+\eps\ueps} \cdot \rho'(\weps)
	\le \io \ueps \rho'(\weps)
	\le \io (\weps+1) \rho'(\weps)
	\qquad \mbox{for all } t\in (0,\tme).
  \ee
  Inserting (\ref{33.3}) and (\ref{33.4}) into (\ref{33.2}) yields (\ref{33.1}).
\qed
\begin{lem}\label{lem04}
  For $\eps\in (0,1)$, let
  \be{04.1}
	\rho_\eps(\xi):=\frac{(\xi+1)^3}{1+\eps\xi},
	\qquad \xi\ge 0.
  \ee
  Then
  \be{04.2}
	0 \le \rho_\eps'(\xi) \le \frac{3(\xi+1)^2}{1+\eps\xi}
	\qquad \mbox{for all } \xi\ge 0
  \ee
  and
  \be{04.3}
	\frac{3(\xi+1)}{2(1+\eps\xi)} \le \rho_\eps''(\xi) \le \frac{8(\xi+1)}{1+\eps\xi}
	\qquad \mbox{for all } \xi\ge 0.
  \ee
\end{lem}
\proof
  We differentiate to see that
  \be{4.23}
	\rho_\eps'(\xi)=\frac{3(\xi+1)^2}{1+\eps\xi} - \frac{\eps(\xi+1)^3}{(1+\eps\xi)^2}
	\qquad \mbox{and} \qquad
	\rho_\eps''(\xi)=\frac{6(\xi+1)}{1+\eps\xi} - \frac{6\eps(\xi+1)^2}{(1+\eps\xi)^2} + \frac{2\eps^2(\xi+1)^3}{(1+\eps\xi)^3}
  \ee
  for all $\xi\ge 0$, and that thus (\ref{04.2}) follows upon observing that
  \bas
	\frac{\eps(\xi+1)^3}{(1+\eps\xi)^2}
	= \frac{(\xi+1)^2}{1+\eps\xi} \cdot \frac{\eps\xi+\eps}{1+\eps\xi}
	\le \frac{(\xi+1)^2}{1+\eps\xi}
	\qquad \mbox{for all } \xi\ge 0.
  \eas
  Since
  \bas
	\frac{6\eps(\xi+1)^2}{(1+\eps\xi)^2}
	\le \frac{2\eps^2(\xi+1)^3}{(1+\eps\xi)^3} + \frac{9(\xi+1)}{2(1+\eps\xi)}
	\qquad \mbox{for all } \xi\ge 0
  \eas
  by Young's inequality, and since
  \bas
	\frac{\frac{2\eps^2(\xi+1)^3}{(1+\eps\xi)^3}}{\frac{6(\xi+1)}{1+\eps\xi}}
	= \frac{\eps^2(\xi+1)^2}{3(1+\eps\xi)^2}
	= \frac{(\eps+\eps\xi)^2}{3(1+\eps\xi)^2}
	\le \frac{1}{3}
	\qquad \mbox{for all } \xi\ge 0,
  \eas
  from (\ref{4.23}) we moerover obtain (\ref{04.3}).
\qed
In conjunction with an inequality describing the evolution of $t\mapsto \io (\veps+1)^3$ for $\eps\in (0,1)$ (see (\ref{4.4}),
a particular version of (\ref{33.1}) can be used to establish a first collection of estimates which can be viewed
as approximate counterparts of corresponding properties addressed in
\cite[Lemma 3.2]{taowin_JMPA} for the unperturbed problem (\ref{0}):
\begin{lem}\label{lem4}
  For every $K>0$ there exists $\Gamma_1(K)>0$ such that if (\ref{init}) and (\ref{K}) hold, then
  \be{4.1}
	\io \Big( \frac{\ueps(\cdot,t)}{1+\eps\ueps(\cdot,t)}\Big)^3 \le \Gamma_1(K)
	\qquad \mbox{for all $t\in (0,\tme)$ and } \eps\in (0,1)
  \ee
  and
  \be{4.00}
	\io \weps^2(\cdot,t) \le \Gamma_1(K)
	\qquad \mbox{for all $t\in (0,\tme)$ and } \eps\in (0,1)
  \ee
  as well as
  \be{4.01}
	\int_t^{t+\tau_\eps} \io |\na\weps|^2  \le \Gamma_1(K)
	\qquad \mbox{for all $t\in (0,\tme-\tau_\eps)$ and } \eps\in (0,1)
  \ee
  and
  \be{4.02}
	\int_t^{t+\tau_\eps} \io |\na\veps|^2  \le \Gamma_1(K)
	\qquad \mbox{for all $t\in (0,\tme-\tau_\eps)$ and } \eps\in (0,1),
  \ee
  where for $\eps\in (0,1)$ we have set $\tau_\eps:=\min\{1,\frac{1}{2}\tme\}$.
\end{lem}
\proof
  For $\eps\in (0,1)$ we let $\rho_\eps$ be as in Lemma \ref{lem04}, and drawing on the left inequalities in (\ref{04.2}) and
  (\ref{04.2}) we may employ Lemma \ref{lem33} to see that, again thanks to (\ref{04.2}) and (\ref{04.3}),
  \bea{4.3}
  \hspace*{-6mm}
	\frac{d}{dt} \io \rho_\eps(\weps)
	&\le& - \frac{D}{2} \io \rho_\eps''(\weps) |\na\weps|^2
	+ \frac{(d-D)^2}{2D} \io \rho_\eps''(\weps) |\na\veps|^2
	+ \io (\weps+1)\rho_\eps'(\weps) \nn\\
	&\le& - \frac{3D}{4} \io \frac{\weps+1}{1+\eps\weps} |\na\weps|^2
	+ \frac{4(d-D)^2}{D} \io \frac{\weps+1}{1+\eps\weps} |\na\veps|^2
	+ 3 \io \frac{(\weps+1)^3}{1+\eps\weps}
  \eea
  for all $t\in (0,\tme)$.
  We next abbreviate $d_0:=\min\{d,1\}$ and observe that since $\ueps\le\weps$ and thus
  $\frac{1}{1+\eps\ueps} \ge \frac{1}{1+\eps\weps}$,
  \bas
	\Big(d+\frac{\ueps}{1+\eps\ueps}\Big) \cdot (\veps+1)
	&\ge& \Big(d+\frac{\ueps}{1+\eps\weps}\Big) \cdot (\veps+1) \\
	&=& \frac{(d+d\eps\weps+\ueps)(\veps+1)}{1+\eps\weps} \\
	&\ge& \frac{(d+\ueps)(\veps+1)}{1+\eps\weps} \\
	&=& \frac{d\veps+d+\ueps\veps+\ueps}{1+\eps\weps} \\
	&\ge& \frac{d\veps+d+\ueps}{1+\eps\weps} \\
	&\ge& \frac{d_0 \veps + d_0 + d_0\ueps}{1+\eps\weps} \\
	&=& d_0 \cdot \frac{\weps+1}{1+\eps\weps}
	\quad \mbox{in } \Om\times (0,\tme).
  \eas
  Testing the seond equation in (\ref{0eps}) by $(\veps+1)^2$ we thus obtain that
  \bea{4.4}
	& & \hs{-20mm}
	\frac{1}{3} \frac{d}{dt} \io (\veps+1)^3 \nn\\
	&=& \io (\veps+1)^2 \na\cdot \Big\{ \Big(d+\frac{\ueps}{1+\eps\ueps}\Big)\na\veps\Big\}
	- \io (\veps+1)^2 \veps
	+ \io (\veps+1)^2 \cdot \frac{\ueps}{1+\eps\ueps} \nn\\
	&=& - 2 \io \Big(d+\frac{\ueps}{1+\eps\ueps}\Big)\cdot (\veps+1) |\na\veps|^2
	- \io (\veps+1)^3
	+ \io (\veps+1)^2 \cdot \Big(1+\frac{\ueps}{1+\eps\ueps}\Big) \nn\\
	&\le& - 2 d_0 \io \frac{\weps+1}{1+\eps\weps} |\na\veps|^2
	- \io (\veps+1)^3
	+ \io (\veps+1)^2 \cdot \Big(1+\frac{\ueps}{1+\eps\ueps}\Big)
  \eea
  for all $t\in (0,\tme)$.
  Here, using that $\frac{d}{d\xi} \frac{\xi}{1+\eps\xi}\ge 0$ for all $\xi\ge 0$ we can estimate
  \bas
	1+\frac{\ueps}{1+\eps\ueps}
	\le 1 + \frac{\weps}{1+\eps\weps}
	= \frac{1+(\eps+1)\weps}{1+\eps\weps}
	\le 2\cdot \frac{\weps+1}{1+\eps\weps}
	\qquad \mbox{in } \Om\times (0,\tme),
  \eas
  and employ Young's inequality to see that, accordingly,
  \bas
	\io (\veps+1)^2 \cdot\Big(1+\frac{\ueps}{1+\eps\ueps}\Big)
	&\le& \frac{2}{3} \io (\veps+1)^3
	+ \frac{1}{3} \io \Big(1+\frac{\ueps}{1+\eps\ueps}\Big)^3 \\
	&\le& \frac{2}{3} \io (\veps+1)^3
	+ \frac{8}{3} \io \Big(\frac{\weps+1}{1+\eps\weps}\Big)^3 \\
	&\le& \frac{2}{3} \io (\veps+1)^3
	+ \frac{8}{3} \io \frac{(\weps+1)^3}{1+\eps\weps}
	\qquad \mbox{for all } t\in (0,\tme).
  \eas
  Writing $b:=\frac{4(d-D)^2+D}{6d_0 D}$ and recalling (\ref{04.1}), from (\ref{4.3}) and (\ref{4.4}) we hence infer that
  \bea{4.5}
	& & \hs{-20mm}
	\frac{d}{dt} \bigg\{
	\io \frac{(\weps+1)^3}{1+\eps\weps}
	+ b\io (\veps+1)^3 \bigg\}
	+ \io \frac{\weps+1}{1+\eps\weps} |\na\veps|^2
	+ \frac{3D}{4} \io \frac{\weps+1}{1+\eps\weps} |\na\weps|^2 \nn\\
	& & \hs{-4mm}
	+ \io \frac{(\weps+1)^3}{1+\eps\weps}
	+ b\io (\veps+1)^3 \nn\\
	&\le& (4+8b) \io \frac{(\weps+1)^3}{1+\eps\weps}
	\qquad \mbox{for all } t\in (0,\tme).
  \eea
  Now Lemma \ref{lem1} says that if we let $c_1:=(4+8b)\Lam_1\Big(\frac{3D}{8(4+8b)}\Big)$ with $\Lam_1(\cdot)$ as found there, then
  \bas
	(4+8b)\io \frac{(\weps+1)^3}{1+\eps\weps}
	\le \frac{3D}{8} \io \frac{\weps+1}{1+\eps\weps} |\na\weps|^2
	+ c_1\cdot \bigg\{ \io (\weps+1)\bigg\}^3
	\qquad \mbox{for all } t\in (0,\tme),
  \eas
  so that since
  \bas
	\io (\weps+1) = \io \ueps + \io \veps + |\Om|
	\le \io u_0 + \max \bigg\{ \io u_0 \, , \, \io v_0\bigg\} + |\Om|
	\le c_2\equiv c_2(K):=2K|\Om| + |\Om|
  \eas
  by (\ref{w}), (\ref{mass}) and (\ref{K}), we conclude that
  \bas
	\yeps(t):=\io \frac{(\weps+1)^3}{1+\eps\weps}
	+ b\io (\veps+1)^3,
	\qquad t\in [0,\tme), \ \eps\in (0,1),
  \eas
  as well as
  \bas
	\heps(t):=
	\io \frac{\weps+1}{1+\eps\weps} |\na\veps|^2
	+ \frac{3D}{8} \io \frac{\weps+1}{1+\eps\weps} |\na\weps|^2
	\qquad t\in (0,\tme), \ \eps\in (0,1),
  \eas
  satisfy
  \be{4.99}
	\yeps'(t) + \yeps(t) + \heps(t)
	\le c_1 c_2^3
	\qquad \mbox{for all } t\in (0,\tme).
  \ee
  Again relying on (\ref{K}), by means of an ODE comparison argument and an integration we obtain from this that	
  \bas
	\yeps(t)
	&\le& \max \bigg\{ c_1 c_2^3 \, , \, \io \frac{(u_0+v_0+1)^3}{1+\eps(u_0+v_0)} + b\io (v_0+1)^3 \bigg\} \\
	&\le& c_4\equiv c_4(K) := \max \Big\{ c_1 c_2^3 \, , \, (2K+1)^3 |\Om| + b\cdot (K+1)^3 |\Om| \Big\}
	\qquad \mbox{for all } t\in [0,\tme),
  \eas
  and that thus, by an integration in (\ref{4.99}),
  \bas
	\int_t^{t+\tau_\eps} \heps(s) ds
	\le \yeps(t) + c_1 c_2^3 \tau_\eps
	\le c_4 + c_1 c_2^3
	\qquad \mbox{for all $t\in (0,\tme-\tau_\eps)$,}
  \eas
  because $\tau_\eps\le 1$ for all $\eps\in (0,1)$.
  According to our definitions of $(\yeps)_{\eps\in (0,1)}$ and $(\heps)_{\eps\in (0,1)}$,
  the claim thus readily follows upon observing that in line with (\ref{04.2}) and the inequality $\ueps\le\weps$ we have
  \bas
	\io \Big(\frac{\ueps}{1+\eps\ueps}\Big)^3
	\le \io \frac{(\ueps+1)^3}{1+\eps\ueps}
	= \io \rho_\eps(\ueps)
	\le \io \rho_\eps(\weps)
	= \io \frac{(\weps+1)^3}{1+\eps\weps}
	\qquad \mbox{for all } t\in (0,\tme),
  \eas
  and that $\frac{\weps+1}{1+\eps\weps} \ge 1$ as well as $\frac{(\weps+1)^3}{1+\eps\weps} \ge \weps^2$
  due to the fact that $\frac{\xi+1}{1+\eps\xi}\ge 1$ for all $\xi\ge 0$.
\qed
The actually most important implication of this section can now be achieved by drawing on the time-independent bound in $L^3(\Om)$
stated in (\ref{4.1}) for the source term $\frac{\ueps}{1+\eps\ueps}$ in the second equation in (\ref{0eps}).
Indeed, Corollary \ref{cor_moser} says that in low-dimensional scenarios in which $\frac{n}{2}<3$,
this is sufficient to ensure $L^\infty$ bounds for the $\veps$:
\begin{lem}\label{lem5}
  For all $K>0$ there exists $M=M(K)>0$ such that if (\ref{init}) and (\ref{K}) are valid, it follows that
  \be{5.1}
	\|\veps(\cdot,t)\|_{L^\infty(\Om)} \le M
	\qquad \mbox{for all $t\in (0,\tme)$ and } \eps\in (0,1).
  \ee
\end{lem}
\proof
  If (\ref{init}) and (\ref{K}) hold, then according to the nonnegativity of the $\veps$, (\ref{0eps}) implies that
  \bas
	v_{\eps t} \le \na\cdot \big(a_\eps(x,t)\na\veps\big) + \feps(x,t)
	\quad \mbox{in } \Om\times (0,\tme)
	\qquad \mbox{for all } \eps\in (0,1),
  \eas
  where $a_\eps:=d+\frac{\ueps}{1+\eps\ueps}$ and $\feps:=\frac{\ueps}{1+\eps\ueps}$ satisfy
  $a_\eps\ge d$, $\feps\ge 0$ and, by (\ref{4.1}),
  \bas
	\io \feps^3 \le \Gamma_1
	\qquad \mbox{for all $t\in (0,\tme)$ and } \eps\in (0,1),
  \eas
  with $\Gamma_1:=\Gamma_1(K)$ and $\Gamma_1(\cdot)$ taken from Lemma \ref{lem4}.
  Using that $3>\frac{n}{2}$, we may thus draw on Corollary \ref{cor_moser} to see with with $\Lam_4(\cdot,\cdot)$ as introduced
  there, thanks to (\ref{mass}) and (\ref{K}) we have
  \bas
	\|\veps(\cdot,t)\|_{L^\infty(\Om)}
	&\le& \Lam_4\Big(3,\max\Big\{\frac{1}{d},\Gamma_1^\frac{1}{3} \Big\} \Big) \cdot
	\max \bigg\{ \|v_0+1\|_{L^\infty(\Om)} \, , \, \sup_{s\in (0,T)} \|v(\cdot,s)+1\|_{L^1(\Om)} \bigg\} \\
	&\le& \Lam_4\Big(3,\max\Big\{\frac{1}{d},\Gamma_1^\frac{1}{3} \Big\} \Big) \cdot (K+1)\max\{1, \, |\Om|\}
  \eas
  for all $t\in (0,\tme)$ and $\eps\in (0,1)$.
\qed
\mysection{Global solvability of (\ref{0eps})}
Unlike in the original problem (\ref{0}), $L^\infty$ bounds for the second components of the solutions to the non-degenerate
regularized variants (\ref{0eps}) already entail higher regularity features.
The key toward this can be verified by reduction to standard literature on parabolic regularity theory:
\begin{lem}\label{lem6}
  Suppose that (\ref{init}) holds, and that $\eps\in (0,1)$ is such that $\tme<\infty$.
  Then there exist $\theta=\theta(\eps,u_0,v_0)\in (0,1)$ and $C(\eps,u_0,v_0)>0$ such that
  \be{6.1}
	\|\veps\|_{C^{\theta,\frac{\theta}{2}}(\bom\times [0,T])}
	\le C(\eps,u_0,v_0)
	\qquad \mbox{for all $T\in (0,\tme)$ and } \eps\in (0,1).
  \ee
\end{lem}
\proof
  The second equation in (\ref{0eps}) can be recast according to
  \bas
	v_{\eps t}= \na\cdot \big( a_\eps(x,t)\na\veps\big) + \feps(x,t),
	\qquad x\in\Om, \ t\in (0,\tme),
  \eas
  with $a_\eps:=d+\frac{\ueps}{1+\eps\ueps}$ and $\feps:=-\veps+\frac{\ueps}{1+\eps\ueps}$ satisfying
  \bas
	d \le a_\eps \le d+\frac{1}{\eps}
	\quad \mbox{and} \quad
	|\feps| \le M + \frac{1}{\eps}
	\qquad \mbox{in } \Om\times (0,\tme)
  \eas
  by Lemma \ref{lem5}, where $M=M(K)$ with $K:=\|u_0\|_{L^\infty(\Om)} + \|v_0\|_{L^\infty(\Om)}$.
  Again explicitly relying on Lemma \ref{lem5}, we therefore obtain the claim as an immediate consequence of
  known results on H\"older regularity of bounded solutions to scalar parabolic problems (\cite{porzio_vespri}).
\qed
When combined with (\ref{4.00}), the information on time-independent H\"older regularity of the $\veps$ contained in (\ref{6.1})
can be seen to imply an $L^\infty$ bound for the first component in (\ref{0eps}):
\begin{lem}\label{lem60}
  Let (\ref{init}) hold, and let $\eps\in (0,1)$ be such that $\tme<\infty$.
  Then there exists $C(\eps,u_0,v_0)>0$ fulfilling
  \be{60.1}
	\|\ueps(\cdot, t)\|_{L^\infty(\Om)}
	\le C(\eps,u_0,v_0)
	\qquad \mbox{for all $t\in (0,\tme)$ and } \eps\in (0,1).
  \ee
\end{lem}
\proof
  Lemma \ref{lem4} guarantees the existence of $c_1=c_1(u_0, v_0)>0$ such that
  \be{60.2}
 	\io \weps^2(\cdot, t) \le c_1
 	\qquad \mbox{for all $t\in (0,\tme)$ and } \eps\in (0,1),
  \ee
  while (\ref{6.1}) provides $c_2=c_2(\eps, u_0, v_0)$ such that
  \bas
 	\|\veps(\cdot, t)\|_{C^\theta(\bom)}\le c_2
 	\qquad \mbox{for all $t\in (0,\tme)$ and } \eps\in (0,1).
  \eas
  Together with \cite[Lemma 2.2]{taowin_JMPA}, the latter entails the existence of $c_3=c_3(D, \theta)>0$ satisfying
  \be{60.3}
	\Bigg\| \int_0^t e^{(t-s)(D\Del-1)} \Del\veps(\cdot,s) ds \Bigg\|_{L^\infty(\Om)}
	\le c_3\cdot\sup_{s\in (0,t)} \|\veps(\cdot,s)\|_{C^\theta(\bom)}
    	\le c_3\cdot c_2
  \ee
  for all $t\in (0,\tme)$ and $\eps\in (0,1)$. Since the first equation in (\ref{0w}) can be rewritten in the form
  \bas
 	w_{\eps t} = D\Del\weps -\weps + (d-D) \Del\veps - \veps + \frac{\ueps}{1+\eps\ueps} +\weps,
	\qquad & x\in\Om, \ t\in (0,\tme),
  \eas
  relying on order preservation of $(e^{\tau\Del})_{\tau\ge 0}$ and noting $\frac{\ueps}{1+\eps\ueps} \le \ueps \le \weps$
  we obtain from an associatd Duhamel representation that
  \bas
	\weps(\cdot,t)
	&=& e^{t(D\Del-1)} w_0
	+ (d-D) \int_0^t e^{(t-s)(D\Del-1)} \Del \veps(\cdot,s) ds \\
	& & - \int_0^t e^{(t-s)(D\Del-1)} \veps(\cdot,s) ds
	+ \int_0^t e^{(t-s)(D\Del-1)} \Big\{\frac{\ueps(\cdot, s)}{1+\eps\ueps(\cdot, s)} +\weps(\cdot,s)\Big\} ds \\
	&\le& \|w_0\|_{L^\infty(\Om)}
	+ (d-D) \int_0^t e^{(t-s)(D\Del-1)} \Del \veps(\cdot,s) ds  + 2\int_0^t e^{(t-s)(D\Del-1)}  \weps(\cdot,s) ds
	\quad \mbox{in } \Om
  \eas
  for all $t\in (0,\tme)$ and $\eps\in (0,1)$, so that from (\ref{60.3}) we infer that for all $t\in (0,\tme)$ and $\eps\in (0,1)$,
  \bea{60.4}
  	\|\weps(\cdot,t)\|_{L^\infty(\Om)}
	&\le &\|w_0\|_{L^\infty(\Om)}
	+ (d-D) \cdot c_2 c_3 \nn\\
    	& &  +2 \int_0^t \Big(1+D^{-\frac{n}{6}}(t-s)^{-\frac{n}{6}}\Big) e^{-(t-s)}\| \weps(\cdot,s)\|_{L^3(\Om)} ds.
  \eea
  For fixed $T\in (0, \tme)$ writing $A_\eps(T):=\max_{t\in [0, T]} \|\weps(\cdot,t)\|_{L^\infty(\Om)}$ and
  $c_4=c_4(\eps, u_0, v_0):=\|u_0+v_0\|_{L^\infty(\Om)} + (d-D) \cdot c_2 c_3$,
  from (\ref{60.4}) and a simple interpolation we obtain that due to (\ref{60.2}),
  \bas
 	\|\weps(\cdot,t)\|_{L^\infty(\Om)}
 	&\le& c_4 +  2\int_0^t \Big(1+D^{-\frac{n}{6}}(t-s)^{-\frac{n}{6}}\Big) e^{-(t-s)}\|\weps(\cdot,s)\|_{L^2(\Om)}^\frac{2}{3}
        \cdot \|\weps(\cdot,s)\|_{L^\infty(\Om)}^\frac{1}{3} ds\\
 	&\le& c_4 +2c_1^\frac{1}{3}c_5\cdot A_\eps^\frac{1}{3}(T)
 	\qquad \mbox{for all } t\in (0, T),
  \eas
  where $c_5:=\int_0^\infty \big(1+D^{-\frac{n}{6}}\sigma^{-\frac{n}{6}}\big) e^{-\sigma} d\sigma$ is finite according to our
  assumption that $n\le 5$.
  In conjunction with Young's inequality, this entails that
  \bas
 	A_\eps(T)\le  c_4 +2c_1^\frac{1}{3}c_5\cdot A_\eps^\frac{1}{3}(T)
  	\le c_4 +\frac{2}{3} A_\eps(T) +\frac{4}{3} c_1^\frac{1}{2}c_5^\frac{3}{2}
  \eas
  and that, consequently,
  \bas
 	A_\eps(T)\le C(\eps, u_0, v_0):=3 c_4 +4 c_1^\frac{1}{2}c_5^\frac{3}{2}
 	\qquad \mbox{for all } t\in (0, T),
  \eas
  which implies (\ref{60.1}) due to the evident fact that
  $\|\weps(\cdot,t)\|_{L^\infty(\Om)} \ge  \|\ueps(\cdot,t)\|_{L^\infty(\Om)}$ for all $t\in (0,\tme)$.
 \qed
Having the above information at hand, we can rearrange the approach developed for (\ref{0}) in
\cite[Lemmata 5.2-5.6]{taowin_JMPA}
to establish bounds for gradients of solutions to (\ref{0eps}) in $L^p$ spaces with arbitrarily large finite $p$.
\begin{lem}\label{lem7}
  If (\ref{init}) holds and $\eps\in (0,1)$ is such that $\tme<\infty$,
  then for each $p\ge 4$ there exists $C(\eps,p,u_0,v_0)>0$ such that
  \be{7.1}
	\io |\na\ueps(\cdot,t)|^p
	+ \io |\na\veps(\cdot,t)|^p
	\le C(\eps,p,u_0,v_0)
	\qquad \mbox{for all $t\in (0,\tme)$ and } \eps\in (0,1).
  \ee
\end{lem}
\proof
  A proof is sketched in an appendix below.
\qed
In consequence, each of our approximate solutions actually is global in time:
\begin{lem}\label{lem8}
  Whenever (\ref{init}) holds, we have $\tme=\infty$ for all $\eps\in (0,1)$.
\end{lem}
\proof
  In view of (\ref{ext}), this directly results from (\ref{mass}) and an application of Lemma \ref{lem7} to any $p\ge 4$
  fulfilling $p>n$.
\qed
\mysection{Constructing a global weak solution of (\ref{0})}
The mere construction of a global weak solution to (\ref{0}) can, in its essence, already be based solely on Lemma \ref{lem4}
and the following fairly straightforward consequence thereof on time regularity.
\begin{lem}\label{lem81}
  Assume (\ref{init}). Then for all $T>0$ there exists $C(T,u_0,v_0)>0$ such that
  \be{81.1}
	\int_0^T \|u_{\eps t}(\cdot,t)\|_{(W^{1,6}(\Om))^\star}^2 dt \le C(T,u_0,v_0)
	\qquad \mbox{for all } \eps\in (0,1)
  \ee
  and
  \be{81.2}
	\int_0^T \|v_{\eps t}(\cdot,t)\|_{(W^{1,6}(\Om))^\star}^2 dt \le C(T,u_0,v_0)
	\qquad \mbox{for all } \eps\in (0,1).
  \ee
\end{lem}
\proof
  For definiteness in notation, we fix $c_1>0$ such that for each $\psi\in C^1(\bom)$ fulfilling $\|\psi\|_{W^{1,6}(\Om)} \le 1$
  we have $\|\na\psi\|_{L^6(\Om)} + \|\na\psi\|_{L^2(\Om)} + \|\psi\|_{L^\frac{3}{2}(\Om)} + \|\psi\|_{L^1(\Om)} \le c_1$.
  For any such $\psi$, recalling that $\ueps=\weps-\veps$ for $\eps\in (0,1)$,
  integrating by parts in (\ref{0eps}), we then obtain that
  \bas
	\bigg| \io u_{\eps t} \psi \bigg|
	&=& \bigg| - D \io \na\weps\cdot\na\psi
	+ \io \Big\{ D + \frac{\ueps}{1+\eps\ueps} \Big\} \na\veps\cdot\na\psi \bigg| \\
	&\le& D \|\na\weps\|_{L^2(\Om)} \|\na\psi\|_{L^2(\Om)}
	+ \Big\|D + \frac{\ueps}{1+\eps\ueps}\Big\|_{L^3(\Om)} \|\na\veps\|_{L^2(\Om)} \|\na\psi\|_{L^6(\Om)} \\
	&\le& c_1 D \|\na\weps\|_{L^2(\Om)}
	+ c_1 \Big\|D + \frac{\ueps}{1+\eps\ueps}\Big\|_{L^3(\Om)} \|\na\veps\|_{L^2(\Om)}
	\qquad \mbox{for all $t>0$ and } \eps\in (0,1)
  \eas
  and thus
  \bea{81.3}
	& & \hs{-24mm}
	\int_0^T \|u_{\eps t}(\cdot,t)\|_{(W^{1,6}(\Om))^\star}^2 dt \nn\\
	&\le& \int_0^T \Big\{
	c_1 D \|\na\weps(\cdot,t)\|_{L^2(\Om)}
	+ c_1 \Big\| D +\frac{\ueps(\cdot,t)}{1+\eps\ueps(\cdot,t)}\Big\|_{L^3(\Om)} \|\na\veps(\cdot,t)\|_{L^2(\Om)}
	\Big\}^2 dt \nn\\
	&\le& 2c_1^2 D^2 \int_0^T \io |\na \weps|^2
	+ 2c_1^2 \cdot \bigg\{ D|\Om|^\frac{1}{3} +
		\sup_{t>0} \Big\|\frac{\ueps(\cdot,t)}{1+\eps\ueps(\cdot,t)}\Big\|_{L^3(\Om)} \bigg\}^2
		\cdot \int_0^T \io |\na\veps|^2
  \eea
  for all $T>0$ and $\eps\in (0,1)$.
  Similarly, (\ref{0eps}) implies that
  for all $t>0$ and $\eps\in (0,1)$,
  \bas
	\bigg| \io v_{\eps t} \psi \bigg|
	&=& \bigg| - \io \Big(d+\frac{\ueps}{1+\eps\ueps}\Big) \na\veps\cdot\na\psi
	- \io \veps\psi
	+ \io \frac{\ueps}{1+\eps\ueps} \psi \bigg| \\
	&\le& \Big\| d+\frac{\ueps}{1+\eps\ueps}\Big\|_{L^3(\Om)} \|\na\veps\|_{L^2(\Om)} \|\na\psi\|_{L^6(\Om)} \\
	& & + \|\veps\|_{L^\infty(\Om)} \|\psi\|_{L^1(\Om)}
	+ \Big\|\frac{\ueps}{1+\eps\ueps}\Big\|_{L^3(\Om)} \|\psi\|_{L^\frac{3}{2}(\Om)} \\
	&\le& c_1 \cdot \Big\{ d |\Om|^\frac{1}{3} + \Big\|\frac{\ueps}{1+\eps\ueps}\Big\|_{L^3(\Om)} \Big\}
		\cdot \|\na\veps\|_{L^2(\Om)}
	+ c_1 \|\veps\|_{L^\infty(\Om)}
	+ c_1 \Big\|\frac{\ueps}{1+\eps\ueps}\Big\|_{L^3(\Om)}
  \eas
  and hence
  \bea{81.4}
	& & \hs{-30mm}
	\int_0^T \|v_{\eps t}(\cdot,t)\|_{(W^{1,6}(\Om))^\star}^2 dt \nn\\
	&\le& 3c_1^2 \cdot \bigg\{ d|\Om|^\frac{1}{3}
		+ \sup_{t>0} \Big\|\frac{\ueps(\cdot,t)}{1+\eps\ueps(\cdot,t)}\Big\|_{L^3(\Om)} \bigg\}^2
		\cdot \int_0^T \io |\na\veps|^2 \nn\\
	& & + 3c_1^2 \cdot \bigg\{ \sup_{t>0} \|\veps(\cdot,t)\|_{L^\infty(\Om)} \bigg\}^2 \cdot T
	+ 3c_1^2 \cdot \bigg\{ \sup_{t>0} \Big\|\frac{\ueps(\cdot,t)}{1+\eps\ueps(\cdot,t)}\Big\|_{L^3(\Om)} \bigg\}^2 \cdot T
  \eea
  for all $T>0$ and $\eps\in (0,1)$.
  In view of Lemma \ref{lem4} and Lemma \ref{lem5}, from (\ref{81.3}) and (\ref{81.4}) we obtain both (\ref{81.1}) and
  (\ref{81.2}) with some suitably large $C(T,u_0,v_0)>0$.
\qed
Indeed, a combination of Lemma \ref{lem4} with Lemma \ref{lem81} yields the following.
\begin{lem}\label{lem82}
  If (\ref{init}) holds, then there exist $(\eps_j)_{j\in\N} \subset (0,1)$ and nonnegative functions $u$ and $v$
  fulfilling (\ref{reg}) such that $\eps_j\searrow 0$ as $j\to\infty$, that
  \begin{eqnarray}
	& & \ueps \to u
	\qquad \mbox{in $L^2_{loc}(\bom\times [0,\infty))$ and a.e.~in } \Om\times (0,\infty),
	\label{82.2} \\
	& & \na\ueps \wto \na u
	\qquad \mbox{in $L^2_{loc}(\bom\times [0,\infty))$,}
	\label{82.3} \\
	& & \veps \to v
	\qquad \mbox{in $L^2_{loc}(\bom\times [0,\infty))$ and a.e.~in } \Om\times (0,\infty),
	\qquad \qquad \mbox{and}
	\label{82.4} \\
	& & \na\veps \wto \na v
	\qquad \mbox{in $L^2_{loc}(\bom\times [0,\infty))$,}
	\label{82.5}
  \end{eqnarray}
  as $\eps=\eps_j\searrow 0$.
  This limit $(u,v)$ forms a global weak solution of (\ref{0}) in the sense of Theorem \ref{theo14}.
\end{lem}
\proof
  For $T>0$, from Lemma \ref{lem4}, (\ref{w}) and (\ref{mass}) we know that
  \bas
	(\ueps)_{\eps\in (0,1)}
	\quad \mbox{and} \quad
	(\veps)_{\eps\in (0,1)}
	\quad \mbox{are bounded in } L^2((0,T);W^{1,2}(\Om)),
  \eas
  while Lemma \ref{lem81} asserts that
  \bas
	(u_{\eps t})_{\eps\in (0,1)}
	\quad \mbox{and} \quad
	(v_{\eps t})_{\eps\in (0,1)}
	\quad \mbox{are bounded in } L^2\big((0,T);(W^{1,6}(\Om))^\star\big).
  \eas
  Two applications of an Aubin-Lions lemma thus yield $(\eps_j)_{j\in\N} \subset (0,1)$ as well as nonnegative elements
  $u$ and $v$ of $L^2_{loc}([0,\infty);W^{1,2}(\Om))$ such that $\eps_j\searrow 0$ as $j\to\infty$, and that
  (\ref{82.2})-(\ref{82.5}) hold as $\eps=\eps_j\searrow 0$.
  For the derivation of (\ref{wu}) and (\ref{wv}), we fix $\vp\in C_0^\infty(\bom\times [0,\infty))$ and then see on integrating
  by parts in (\ref{0eps}) that
  \be{82.6}
	- \int_0^\infty \io \ueps\vp_t - \io u_0 \vp(\cdot,0)
	= -D \int_0^\infty \io \na\ueps\cdot\na\vp
	+ \int_0^\infty \io \frac{\ueps}{1+\eps\ueps} \na\veps\cdot\na\vp
  \ee
  as well as
  \bea{82.7}
	- \int_0^\infty \io \veps\vp_t - \io v_0\vp(\cdot,0)
	&=& - d \int_0^\infty \io \na\veps\cdot\na\vp
	- \int_0^\infty \io \frac{\ueps}{1+\eps\ueps} \na\veps\cdot\na\vp \nn\\
	& & - \int_0^\infty \io \veps\vp
	+ \int_0^\infty \io \frac{\ueps}{1+\eps\ueps} \vp
  \eea
  for all $\eps\in (0,1)$.
  Now from (\ref{82.2}) and the dominated convergence theorem it readily follows that $\frac{\ueps}{1+\eps\ueps} \to u$
  in $L^2_{loc}(\bom\times [0,\infty))$ as $\eps=\eps_j\searrow 0$, which combind with (\ref{82.5}) shows that not only
  \bas
	\int_0^\infty \io \frac{\ueps}{1+\eps\ueps} \vp \to \int_0^\infty \io u\vp
	\qquad \mbox{as $\eps=\eps_j\searrow 0$,}
  \eas
  but also
  \bas
	\int_0^\infty \io \frac{\ueps}{1+\eps\ueps} \na\veps\cdot\na\vp
	\to \int_0^\infty \io u\na v\cdot\na\vp
	\qquad \mbox{as $\eps=\eps_j\searrow 0$.}
  \eas
  Since clearly
  \bas
	\int_0^\infty \io \ueps\vp_t \to \int_0^\infty \io u\vp_t,
	\quad
	\int_0^\infty \io \veps\vp_t \to \int_0^\infty \io v\vp_t
	\quad \mbox{and} \quad
	\int_0^\infty \io \veps\vp \to \int_0^\infty \io v\vp
  \eas
  as $\eps=\eps_j\searrow 0$ by (\ref{82.2}) and (\ref{82.4}), and since moreover
  \bas
	\int_0^\infty \io \na\ueps\cdot\na\vp \to \int_0^\infty \io \na u\cdot\na\vp
	\quad \mbox{and} \quad
	\int_0^\infty \io \na\veps\cdot\na\vp \to \int_0^\infty \io \na v\cdot\na\vp
	\qquad \mbox{as $\eps=\eps_j\searrow 0$}
  \eas
  due to (\ref{82.3}) and (\ref{82.5}), from (\ref{82.6}) and (\ref{82.7}) we infer that indeed both (\ref{wu}) and (\ref{wv}) hold.
\qed
\mysection{Controlling $\ueps$ in exponential Orlicz classes. Proof of Theorem \ref{theo14}}
\subsection{An approximate variant of $\io (w+1)e^{(w+1)^\al}$ and its evolution}
Next approaching the core of our analysis, in this part we will address an approximate counterpart of the Orlicz class
estimate in (\ref{14.1}).
A first step toward this will rely on the outcome of Lemma \ref{lem33} when applied to the functions introduced
and characterized as follows.
\begin{lem}\label{lem333}
  Let $\al>0$ and $\eps\in (0,1)$ be such that
  \be{333.1}
	\eps^\al \le \frac{\al}{2}.
  \ee
  Then for
  \be{333.2}
	\rho_\eps(\xi):=\frac{\xi+1}{1+\eps\xi} \cdot e^{(\xi+1)^\al},
	\qquad \xi\ge 0,
  \ee
  we have
  \be{333.02}
	\rho_\eps(\xi) \le \al \cdot \frac{(\xi+1)^{\al+1}}{1+\eps\xi} \cdot e^{(\xi+1)^\al}
	+ \Big(\frac{e}{\al}\Big)^\frac{1}{\al}
	\qquad \mbox{for all } \xi\ge 0
  \ee
  and
  \be{333.3}
	0 \le (\xi+1)\rho_\eps'(\xi)
	\le 2\al\cdot \frac{(\xi+1)^{\al+1}}{1+\eps\xi} \cdot e^{(\xi+1)^\al}
	+ \Big(\frac{e}{\al}\Big)^\frac{1}{\al}
	\qquad \mbox{for all } \xi\ge 0
  \ee
  as well as
  \be{333.4}
	\frac{\al^2}{2} \cdot \frac{(\xi+1)^{2\al-1}}{1+\eps\xi} \cdot e^{(\xi+1)^\al}
	\le \rho_\eps''(\xi)
	\le 3\al^2 \cdot \frac{(\xi+1)^{2\al-1}}{1+\eps\xi} \cdot e^{(\xi+1)^\al}
	+ e^\frac{1}{\al}
	\qquad \mbox{for all } \xi\ge 0.
  \ee
\end{lem}
\proof
  Using that $\frac{d}{d\xi} \frac{\xi+1}{1+\eps\xi}=\frac{1-\eps}{(1+\eps\xi)^2}$ for all $\xi\ge 0$, we calculate
  \be{333.5}
	\rho_\eps'(\xi)
	= \al \cdot \frac{(\xi+1)^\al}{1+\eps\xi} \cdot e^{(\xi+1)^\al}
	+ (1-\eps) \cdot \frac{1}{(1+\eps\xi)^2} \cdot e^{(\xi+1)^\al}
  \ee
  and
  \bea{333.6}
	\rho_\eps''(\xi)
	&=& \al^2 \cdot \frac{(\xi+1)^{2\al-1}}{1+\eps\xi} \cdot e^{(\xi+1)^\al}
	+ \al^2 \cdot \frac{(\xi+1)^{\al-1}}{1+\eps\xi} \cdot e^{(\xi+1)^\al}
	- \al\eps \cdot \frac{(\xi+1)^\al}{(1+\eps\xi)^2} \cdot e^{(\xi+1)^\al} \nn\\
	& & + \al (1-\eps) \cdot \frac{(\xi+1)^{\al-1}}{(1+\eps\xi)^2} \cdot e^{(\xi+1)^\al}
	- 2\eps(1-\eps)\cdot \frac{1}{(1+\eps\xi)^3} \cdot e^{(\xi+1)^\al}
  \eea
  for $\xi\ge 0$, and observe that if $\xi\ge 0$ is such that
  \be{333.7}
	(\xi+1)^\al \ge \frac{1}{\al},
  \ee
  then
  \bas
	\frac{(1-\eps) \cdot \frac{1}{(1+\eps\xi)^2} \cdot e^{(\xi+1)^\al}}
		{\al\cdot \frac{(\xi+1)^\al}{1+\eps\xi} \cdot e^{(\xi+1)^\al}}
	= \frac{1-\eps}{\al} \cdot \frac{1}{(\xi+1)^\al (1+\eps\xi)}
	\le \frac{1}{\al} \cdot \frac{1}{(\xi+1)^\al}
	\le 1
  \eas
  and thus
  \bas
	(\xi+1) \rho_\eps'(\xi) \le 2\al\cdot \frac{(\xi+1)^{\al+1}}{1+\eps\xi} \cdot e^{(\xi+1)^\al}.
  \eas
  If $\xi\ge 0$ is such that (\ref{333.7}) does not hold, however, then
  \bas
	(\xi+1) \cdot (1-\eps) \cdot \frac{1}{(1+\eps\xi)^2} \cdot e^{(\xi+1)^\al}
	\le (\xi+1) e^{(\xi+1)^\al}
	\le \Big(\frac{1}{\al}\Big)^\frac{1}{\al} \cdot e^\frac{1}{\al},
  \eas
  so that (\ref{333.3}) follows due to our assumption that $\eps<1$.\abs
  Likewise, for $\xi\ge 0$ fulfilling (\ref{333.7}) we can estimate
  \bas
	\frac{\rho_\eps(\xi)}{\al\cdot\frac{(\xi+1)^{\al+1}}{1+\eps\xi} \cdot e^{(\xi+1)^\al}}
	= \frac{1}{\al} \cdot \frac{1}{(\xi+1)^\al} \le 1,
  \eas
  while for $\xi\ge 0$ satisfying $(\xi+1)^\al < \frac{1}{\al}$ we have
  \bas
	\rho_\eps(\xi) \le (\xi+1) e^{(\xi+1)^\al} \le \Big(\frac{1}{\al}\Big)^\frac{1}{\al} \cdot e^\frac{1}{\al},
  \eas
  meaning that also (\ref{333.02}) holds.\abs
  In quite a similar fashion, for $\xi\ge 0$ we see that if (\ref{333.7}) is valid, then
  \bas
	\frac{\al(1-\eps)\cdot\frac{(\xi+1)^{\al-1}}{(1+\eps\xi)^2} \cdot e^{(\xi+1)^\al}}
		{\al^2 \cdot \frac{(\xi+1)^{2\al-1}}{1+\eps\xi} \cdot e^{(\xi+1)^\al}}
	= \frac{1-\eps}{\al} \cdot \frac{1}{(\xi+1)^\al (1+\eps\xi)}
	\le \frac{1}{\al} \cdot \frac{1}{(\xi+1)^\al}
	\le 1,
  \eas
  whereas otherwise,
  \bas
	\al(1-\eps) \cdot \frac{(\xi+1)^{\al-1}}{(1+\eps\xi)^2} \cdot e^{(\xi+1)^\al}
	\le \al(\xi+1)^\al e^{(\xi+1)^\al}
	\le \al\cdot \frac{1}{\al} \cdot e^\frac{1}{\al}
	= e^\frac{1}{\al}.
  \eas
  Since moreover $(\xi+1)^{\al-1} \le (\xi+1)^{2\al-1}$ for all $\xi\ge 0$ and thus
  \bas
	\al^2 \cdot \frac{(\xi+1)^{\al-1}}{1+\eps\xi} \cdot e^{(\xi+1)^\al}
	\le \al^2 \cdot \frac{(\xi+1)^{2\al-1}}{1+\eps\xi} \cdot e^{(\xi+1)^\al}
	\qquad \mbox{for all } \xi\ge 0,
  \eas
  again relying on the fact that $\eps\in (0,1)$ we therefore obtain the right inequality in (\ref{333.4}) from (\ref{333.6}).\abs
  The claimed lower bound for $\rho_\eps''$, finally, can be verified by making use of our restriction in (\ref{333.1}),
  which namely asserts that
  \bas
	\frac{\al\eps\cdot \frac{(\xi+1)^\al}{(1+\eps\xi)^2} \cdot e^{(\xi+1)^\al}}
		{\al^2 \cdot \frac{(\xi+1)^{2\al-1}}{1+\eps\xi} \cdot e^{(\xi+1)^\al}}
	&=& \frac{\eps}{\al} \cdot \frac{(\xi+1)^{1-\al}}{1+\eps\xi} \\
	&=& \frac{\eps^\al}{\al} \cdot \frac{(\eps\xi+\eps)^{1-\al}}{(1+\eps\xi)^{1-\al}} \cdot \frac{1}{(1+\eps\xi)^\al} \\
	&\le& \frac{\eps^\al}{\al}
	\le \frac{1}{2}
	\qquad \mbox{for all } \xi\ge 0,
  \eas
  and that, similarly,
  \bas
	\frac{2\eps(1-\eps)\cdot\frac{1}{(1+\eps\xi)^3} \cdot e^{(\xi+1)^\al}}
		{\al(1-\eps)\cdot \frac{(\xi+1)^{\al-1}}{(1+\eps\xi)^2} \cdot e^{(\xi+1)^\al}}
	= 2\cdot\frac{\eps}{\al} \cdot \frac{(\xi+1)^{1-\al}}{1+\eps\xi}
	\le 1
	\qquad \mbox{for all } \xi\ge 0.
  \eas
  Consequently, (\ref{333.6}) implies that indeed also the left inequality in (\ref{333.4}) holds.
\qed
Collecting the above list of inequalities shows that the general evolution property from Lemma \ref{lem33} can be turned
into the following starting point of our analysis toward (\ref{14.1}).
\begin{lem}\label{lem34}
  If $\al>0$ and $\eps\in (0,1)$ is such that $\eps^\al\le\frac{\al}{2}$, then
  \bea{34.1}
	& & \hs{-20mm}
	\frac{d}{dt} \io \frac{\weps+1}{1+\eps\weps} \cdot e^{(\weps+1)^\al}
	+ \io \frac{\weps+1}{1+\eps\weps} \cdot e^{(\weps+1)^\al}
	+ \frac{D\al^2}{4} \io \frac{(\weps+1)^{2\al-1}}{1+\eps\weps} \cdot e^{(\weps+1)^\al} |\na\weps|^2 \nn\\
	&\le& \frac{3(d-D)^2 \al^2}{2D} \io \frac{(\weps+1)^{2\al-1}}{1+\eps\weps} \cdot e^{(\weps+1)^\al} |\na\veps|^2
	+ \frac{(d-D)^2 e^\frac{1}{\al}}{2D} \io |\na\veps|^2 \nn\\
	& & + 3\al \io \frac{(\weps+1)^{\al+1}}{1+\eps\weps} \cdot e^{(\weps+1)^\al}
	+ 2\cdot\Big(\frac{e}{\al}\Big)^\frac{1}{\al} \cdot |\Om|
  \eea
  for all $t>0$.
\end{lem}
\proof
  We let $\rho_\eps$ be as in Lemma \ref{lem333}, and note that then
  \bas
	\io (\weps+1) \rho_\eps'(\weps)
	\le 2\al \io \frac{(\weps+1)^{\al+1}}{1+\eps\weps} \cdot e^{(\weps+1)^\al}
	+ \Big(\frac{e}{\al}\Big)^\frac{1}{\al} \cdot |\Om|
	\qquad \mbox{for all } t>0
  \eas
  by (\ref{333.3}), and that
  \bas
	\io \frac{\weps+1}{1+\eps\weps} \cdot e^{(\weps+1)^\al}
	= \io \rho_\eps(\weps)
	\le \al \io \frac{(\weps+1)^{\al+1}}{1+\eps\weps} \cdot e^{(\weps+1)^\al}
	+ \Big(\frac{e}{\al}\Big)^\frac{1}{\al} \cdot |\Om|
	\qquad \mbox{for all } t>0
  \eas
  thanks to (\ref{333.02}).
  Therefore, (\ref{34.1}) is a consequence of Lemma \ref{lem33} when combined with (\ref{333.4}).
\qed
\subsection{A functional involving a multiplicative coupling of $\veps$ and $\weps$}
Now the key step will aim at an appropriate control of the first integral on the right-hand side of (\ref{34.1}),
viewed here as an expression in the flavor of a Dirichlet integral over $\veps$ that involves a weight function depending
on $\weps$ in a rapidly growing manner.
Our approach toward a compensation of this will be based on an analysis of
functionals which for $\eps\in (0,1)$ couple $\veps$ to $\weps$ in a certain multiplicative manner, 
allowing for some superalgebraic dependencies on $\weps$.
An initial observation in this direction will be formulated in Lemma \ref{lem11}, making use of the simple two-sided
estimate on the effective diffusion rate in the second equation in (\ref{0eps}).
\begin{lem}\label{lem9}
  If $K>0$ and (\ref{init}) as well as (\ref{K}) hold, then
  \be{9.1}
	\frac{\min\{d,1\}}{M+1} \cdot \frac{\weps +1}{1+\eps\weps}
	\le d + \frac{\ueps}{1+\eps\ueps}
	\le (d+1)\cdot \frac{\weps +1}{1+\eps\weps}
	\quad \mbox{in } \Om\times (0,\infty)
	\qquad \mbox{for all } \eps\in (0,1),
  \ee
  where $M=M(K)$ is as in Lemma \ref{lem5}.
\end{lem}
\proof
  Let $\eps\in (0,1)$. Then again since $\frac{d}{d\xi} \frac{\xi}{1+\eps\xi} \ge 0$ for all $\xi\ge 0$,
  the fact that $\ueps\le \weps$ implies that
  \bas
	d + \frac{\ueps}{1+\eps\ueps}
	&\le& d + \frac{\weps}{1+\eps\weps}
	= \frac{d+d\eps\weps+\weps}{1+\eps\weps}
	\le \frac{d+1+d\weps+\weps}{1+\eps\weps} \\
	&=& (d+1) \cdot \frac{\weps +1}{1+\eps\weps}
	\qquad \mbox{in } \Om\times (0,\infty),
  \eas
  from which the right inequality in (\ref{9.1}) follows.\abs
  We next rely on Lemma \ref{lem5}, which namely asserts that once more writing $d_0:=\min\{d,1\}$ we have $\weps+1\le\ueps+M+1$
  and hence, by nonnegativity of $\weps$ and $\ueps$,
  \bas
	\frac{d+\frac{\ueps}{1+\eps\weps}}{\frac{\weps+1}{1+\eps\weps}}
	&=& \frac{d+d\eps\weps+\ueps}{\weps+1}
	\ge \frac{d+\ueps}{\ueps+M+1}
	\ge \frac{d_0+d_0\ueps}{\ueps+M+1}
	= d_0 - \frac{d_0 M}{\ueps+M+1} \\
	&\ge& d_0 - \frac{d_0 M}{M+1}
	= \frac{d_0}{M+1}
	\qquad \mbox{in } \Om\times (0,\infty).
  \eas
  This implies the left inequality in (\ref{9.1}), because clearly $d+\frac{\ueps}{1+\eps\ueps} \ge d + \frac{\ueps}{1+\eps\weps}$
  in $\Om\times (0,\infty)$.
\qed
Relying on the latter in its technical part, the following lemma records the outcome of a procedure which in its essence
can be viewed as consisting in a multiplication of the second equation in (\ref{0eps}) by the product of $\veps$ with
a function $\chi(\weps)$. It turns out that if here $\chi$ satisfies a growth condition mild enough so as to be satisfied
by functions of the form $0\le\xi\mapsto e^{(\xi+1)^\al}$ for small $\al>0$, then effects due to the cross-diffusive
action expressed in (\ref{0w}) can be limited to the appearance of integrals exclusively involving $\weps$ and its gradient:
\begin{lem}\label{lem11}
  Let $K>0$. Then there exist $\al_0(K)\in (0,1]$, $\gamma(K)>0$ and $\Gamma(K)>0$ such that whenever (\ref{init}) and (\ref{K})
  hold and $\chi\in C^2([0,\infty))$ is such that $\chi>0$ on $[0,\infty)$ as well as
  \be{11.01}
	0 \le \chi'(\xi) \le \al_0(K) \cdot \chi(\xi)
	\qquad \mbox{for all } \xi\ge 0,
  \ee
  we have
  \bea{11.1}
	& & \hs{-20mm}
	\frac{d}{dt} \io \veps^2 \chi(\weps)
	+ \gamma(K) \io \frac{\weps+1}{1+\eps\weps} \chi(\weps) |\na\veps|^2
	+ 2 \io \veps^2 \chi(\weps) \nn\\
	&\le& \Gamma(K) \io \frac{\weps+1}{1+\eps\weps} \cdot \frac{\chi'^2(\weps)+ \chi''^2(\weps)}{\chi(\weps)} \cdot |\na\weps|^2
	+ \Gamma(K) \int_{\{\chi''(\weps)<0\}} |\chi''(\weps)| \cdot |\na\weps|^2 \nn\\
	& & + \Gamma(K) \io \frac{\weps +1}{1+\eps\weps} \cdot \chi(\weps)
  \eea
  for all $t>0$ and $\eps\in (0,1)$.
\end{lem}
\proof
  Given $K>0$, we let $M=M(K)>0$ be as in Lemma \ref{lem5}, and defining
  \be{11.2}
	\gamma\equiv \gamma(K):=\frac{\min\{d,1\}}{4\cdot (M+1)}
  \ee
  we choose $\al_0=\al_0(K)\in (0,1]$ in such a way that
  \be{11.22}
	2|d-D| \cdot M \cdot \al_0 \le \gamma.
  \ee
  Then assuming that (\ref{init}) and (\ref{K}) are valid, and that $\chi\in C^2([0,\infty))$ is positive 
  and satisfies
  (\ref{11.01}), for fixed $\eps\in (0,1)$ we integrate by parts using (\ref{0eps}) and (\ref{0w}) to compute
  \bea{11.3}
	\frac{d}{dt} \io \veps^2 \chi(\weps)
	&=& 2 \io \veps \chi(\weps) \na\cdot \Big\{ \Big(d+\frac{\ueps}{1+\eps\ueps}\Big)\na\veps\Big\}
	+ 2 \io \veps \chi(\weps) \cdot \Big\{ - \veps + \frac{\ueps}{1+\eps\ueps}\Big\} \nn\\
	& & + \io \veps^2 \chi'(\weps) \na\cdot \big\{ D\na\weps+(d-D)\na\veps\big\}
        + \io \veps^2 \chi'(\weps) \cdot \Big\{ - \veps + \frac{\ueps}{1+\eps\ueps}\Big\} \nn\\
	&=& - 2 \io \Big(d+\frac{\ueps}{1+\eps\ueps}\Big) \chi(\weps) |\na \veps|^2
	- 2 \io \Big(d+\frac{\ueps}{1+\eps\ueps}\Big) \veps \chi'(\weps) \na\veps\cdot\na\weps \nn\\
	& & - 2 \io \veps^2 \chi(\weps)
	+ 2 \io \frac{\ueps}{1+\eps\ueps} \veps \chi(\weps) \nn\\
	& & -2D \io \veps \chi'(\weps) \na\veps\cdot\na\weps
	- 2(d-D) \io \veps \chi'(\weps) |\na\veps|^2 \nn\\
	& & - D \io \veps^2 \chi''(\weps) |\na\weps|^2
	- (d-D) \io \veps^2 \chi''(\weps) \na\veps\cdot\na\weps \nn\\
    & & - \io \veps^3 \chi'(\weps)+  \io \frac{\ueps}{1+\eps\ueps}\veps^2 \chi'(\weps)
    \qquad \mbox{for all } t>0.
  \eea
  Here in view of the positivity of $\chi$, Young's inequality together with Lemma \ref{lem5} and the right inequality in (\ref{9.1})
  guarantees that
  \bas
	& & \hs{-16mm}
	- 2 \io \Big(d+\frac{\ueps}{1+\eps\ueps}\Big) \veps \chi'(\weps) \na\veps\cdot\na\weps \nn\\
	&\le& \io \Big(d+\frac{\ueps}{1+\eps\ueps}\Big) \chi(\weps) |\na\veps|^2
	+ \io \Big(d+\frac{\ueps}{1+\eps\ueps}\Big) \veps^2 \cdot \frac{\chi'^2(\weps)}{\chi(\weps)} \cdot |\na\weps|^2 \nn\\
	&\le& \io \Big(d+\frac{\ueps}{1+\eps\ueps}\Big) \chi(\weps) |\na\veps|^2
	+ (d+1)M^2 \io \frac{\weps+1}{1+\eps\weps} \cdot \frac{\chi'^2(\weps)}{\chi(\weps)} \cdot |\na\weps|^2
	\qquad \mbox{for all } t>0,
  \eas
  while thanks to the left inequality in (\ref{9.1}),
  \bas
	\io \Big(d+\frac{\ueps}{1+\eps\ueps}\Big) \chi(\weps) |\na\weps|^2
	\ge 4\gamma \io \frac{\weps+1}{1+\eps\weps} \chi(\weps) |\na\veps|^2
	\qquad \mbox{for all } t>0.
  \eas
  As $\chi'\ge 0$ and
  \bas
	- D \io \veps^2 \chi''(\weps) |\na\weps|^2
	\le D M^2 \int_{\{\chi''(\weps)<0\}} |\chi''(\weps)| \cdot |\na\weps|^2
	\qquad \mbox{for all } t>0
  \eas
  by Lemma \ref{lem5}, from (\ref{11.3}) we thus obtain that
  \bea{11.4}
	& & \hs{-20mm}
	\frac{d}{dt} \io \veps^2 \chi(\weps)
	+ 4\gamma \io \frac{\weps+1}{1+\eps\weps} \chi(\weps) |\na\veps|^2
	+ 2 \io \veps^2 \chi(\weps) \nn\\
	&\le& (d+1)M^2 \io \frac{\weps+1}{1+\eps\weps} \cdot \frac{\chi'^2(\weps)}{\chi(\weps)} \cdot |\na\weps|^2
	+ D M^2 \int_{\{\chi''(\weps)<0\}} |\chi''(\weps)| \cdot |\na\weps|^2 \nn\\
	& & -2D \io \veps \chi'(\weps) \na\veps\cdot\na\weps
	- 2(d-D) \io \veps \chi'(\weps) |\na\veps|^2 \nn\\
	& &  - (d-D) \io \veps^2 \chi''(\weps) \na\veps\cdot\na\weps
	+ 2 \io \frac{\ueps}{1+\eps\ueps} \veps \chi(\weps)
     +  \io \frac{\ueps}{1+\eps\ueps}\veps^2 \chi'(\weps)
  \eea 
  for all $t>0$,
  and here two further applications of Young's inequality show that again due to Lemma \ref{lem5},
  \bea{11.5}
	& & \hs{-30mm}
	-2D \io \veps \chi'(\weps) \na\veps\cdot\na\weps \nn\\
	&\le& \gamma \io \frac{\weps+1}{1+\eps\weps} \chi(\weps) |\na\veps|^2
	+ \frac{D^2}{\gamma} \io \veps^2 \cdot \frac{1+\eps\weps}{\weps+1} \cdot \frac{\chi'^2(\weps)}{\chi(\weps)} \cdot |\na\weps|^2
		\nn\\
	&\le& \gamma \io \frac{\weps+1}{1+\eps\weps} \chi(\weps) |\na\veps|^2
	+ \frac{D^2 M^2}{\gamma} \io \frac{1+\eps\weps}{\weps+1} \cdot \frac{\chi'^2(\weps)}{\chi(\weps)} \cdot |\na\weps|^2
  \eea
  and
  \bea{11.6}
	& & \hs{-30mm}
	- (d-D) \io \veps^2 \chi''(\weps) \na\veps\cdot\na\weps \nn\\
	&\le& \gamma \io \frac{\weps+1}{1+\eps\weps} \chi(\weps) |\na\veps|^2
	+ \frac{(d-D)^2}{4\gamma} \io \veps^4 \cdot \frac{1+\eps\weps}{\weps+1} \cdot \frac{\chi''^2(\weps)}{\chi(\weps)}
		\cdot |\na\weps|^2 \nn\\
	&\le& \gamma \io \frac{\weps+1}{1+\eps\weps} \chi(\weps) |\na\veps|^2
	+ \frac{(d-D)^2 M^4}{4\gamma} \io \frac{1+\eps\weps}{\weps+1} \cdot \frac{\chi''^2(\weps)}{\chi(\weps)} \cdot |\na\weps|^2
  \eea
  for all $t>0$. Noting that $\frac{1+\eps\xi}{\xi+1} \le 1$ for all $\xi\ge 0$, we may estimate
  \be{11.66}
	\frac{1+\eps\weps}{\weps+1} \le 1 \le \frac{\weps+1}{1+\eps\weps}
	\qquad \mbox{in } \Om\times (0,\infty)
  \ee
  to see that (\ref{11.5}) and (\ref{11.6}) imply that for all $t>0$,
  \bea{11.7}
	& & \hs{-30mm}
	-2D \io \veps \chi'(\weps) \na\veps\cdot\na\weps
	- (d-D) \io \veps^2 \chi''(\weps) \na\veps\cdot\na\weps \nn\\
	&\le& 2 \gamma \io \frac{\weps+1}{1+\eps\weps} \chi(\weps) |\na\veps|^2
	+ \frac{D^2 M^2}{\gamma} \io \frac{\weps+1}{1+\eps\weps} \cdot \frac{\chi'^2(\weps)}{\chi(\weps)} \cdot |\na\weps|^2 \nn\\
	& & + \frac{(d-D)^2 M^4}{4\gamma} \io \frac{\weps+1}{1+\eps\weps} \cdot \frac{\chi''^2(\weps)}{\chi(\weps)} \cdot |\na\weps|^2.
  \eea
  Apart from that, we may control the fourth
   to last summand in (\ref{11.4}) by combining (\ref{11.01}) with (\ref{11.22}),
  according to which, namely, it follows that again due to Lemma \ref{lem5} and (\ref{11.66}),
  \bea{11.8}
	- 2(d-D) \io \veps \chi'(\weps) |\na\veps|^2
	&\le& 2|d-D| \cdot M \io \chi'(\weps) |\na\veps|^2 \nn\\
	&\le& 2|d-D| \cdot M \cdot \al_0 \io \chi(\weps) |\na\veps|^2 \nn\\
	&\le& 2|d-D| \cdot M \cdot \al_0 \io \frac{\weps+1}{1+\eps\weps} \cdot \chi(\weps) |\na\veps|^2 \nn\\
	&\le& \gamma \io \frac{\weps+1}{1+\eps\weps} \cdot \chi(\weps) |\na\veps|^2
	\qquad \mbox{for all } t>0.
  \eea
  Since, finally,
  \bas
	2 \io \frac{\ueps}{1+\eps\ueps} \veps \chi(\weps)
	+  \io \frac{\ueps}{1+\eps\ueps}\veps^2 \chi'(\weps)
	&\le& (2M +\al _0 M^2) \io \frac{\weps}{1+\eps\weps} \cdot \chi(\weps) \\
	&\le& (2M +\al _0 M^2) \io \frac{\weps+1}{1+\eps\weps} \cdot \chi(\weps)
	\qquad \mbox{for all } t>0
  \eas
  by Lemma \ref{lem5}
  and (\ref{11.01}) as well as
  the upward monotonicity of $0\le\xi\mapsto \frac{\xi}{1+\eps\xi}$, from (\ref{11.4}), (\ref{11.7})
  and (\ref{11.8}) we readily infer that (\ref{11.1}) holds if we let
  $\Gamma(K):=\max \Big\{ (d+1)M^2 + \frac{D^2 M^2}{\gamma}, \, \frac{(d-D)^2 M^4}{4\gamma}, \, D M^2,$
  $ 2M +\al_0 M^2 \Big\}$.
\qed
Now the particular structure of the first integral on the right of (\ref{34.1}) suggests to here choose the function $\chi$
to be a member of the family characterized in the following lemma.
\begin{lem}\label{lem35}
  Let $\al\in (0,1]$ and
  \be{35.1}
	\chi(\xi):=e^{(\xi+1)^\al},
	\qquad \xi\ge 0.
  \ee
  Then
  \be{35.2}
	\chi'(\xi)=\al(\xi+1)^{\al-1} e^{(\xi+1)^\al}
	\qquad \mbox{for all } \xi\ge 0
  \ee
  and
  \be{35.3}
	\frac{\al^2}{2} (\xi+1)^{2\al-2} e^{(\xi+1)^\al} - 2e^\frac{2}{\al}
	\le \chi''(\xi) \le \al^2(\xi+1)^{2\al-1} e^{(\xi+1)^\al}
	\qquad \mbox{for all } \xi\ge 0,
  \ee
  and for each $\eps\in (0,1)$ we have
  \be{35.4}
	\frac{\xi+1}{1+\eps\xi} \cdot \frac{\chi'^2(\xi)+\chi''^2(\xi)}{\chi(\xi)}
	\le 2\al^2 \cdot \frac{(\xi+1)^{2\al-1}}{1+\eps\xi} \cdot e^{(\xi+1)^\al}
	+ 4e^\frac{2}{\al}
	\qquad \mbox{for all } \xi\ge 0
  \ee
  and
  \be{35.5}
	\frac{\xi+1}{1+\eps\xi} \cdot \chi(\xi)
	\le \al \cdot \frac{(\xi+1)^{\al+1}}{1+\eps\xi} \cdot e^{(\xi+1)^\al}
	+ \Big(\frac{e}{\al}\Big)^\frac{1}{\al}
	\qquad \mbox{for all } \xi\ge 0
  \ee
  as well as
  \be{35.6}
	\frac{\xi+1}{1+\eps\xi} \cdot \chi'(\xi)
	\le \al \cdot \frac{(\xi+1)^{\al+1}}{1+\eps\xi} \cdot e^{(\xi+1)^\al}
	\qquad \mbox{for all } \xi\ge 0.
  \ee
\end{lem}
\proof
  Differentiating in (\ref{35.1}) yields (\ref{35.2}) as well as the identity
  \be{35.7}
	\chi''(\xi)
	= \al^2 (\xi+1)^{2\al-2} e^{(\xi+1)^\al}
	- \al(1-\al) (\xi+1)^{\al-2} e^{(\xi+1)^\al}
	\qquad \mbox{for all } \xi\ge 0,
  \ee
  where in the case when $\xi\ge 0$ satisfies $(\xi+1)^\al\ge\frac{2}{\al}$, we see that
  \bas
	\frac{\al(1-\al)(\xi+1)^{\al-2} e^{(\xi+1)^\al}}{\al^2 (\xi+1)^{2\al-2} e^{(\xi+1)^\al}}
	= \frac{1-\al}{\al} \cdot \frac{1}{(\xi+1)^\al}
	\le \frac{1-\al}{2} \le \frac{1}{2}
  \eas
  and hence, in particular,
  \be{35.8}
	\frac{\al^2}{2} (\xi+1)^{2\al-2} e^{(\xi+1)^\al}
	\le \chi''(\xi) \le \al^2 (\xi+1)^{2\al-2} e^{(\xi+1)^\al}
	\qquad \mbox{for all $\xi\ge 0$ fulfilling $(\xi+1)^\al \ge \frac{2}{\al}$.}
  \ee
  On the other hand, for any $\xi\ge 0$ satisfying $(\xi+1)^\al<\frac{2}{\al}$ we have
  \be{35.9}
	\al(1-\al)(\xi+1)^{\al-2} e^{(\xi+1)^\al}
	\le \al(\xi+1)^\al e^{(\xi+1)^\al}
	\le \al\cdot\frac{2}{\al} \cdot e^\frac{2}{\al} = 2e^\frac{2}{\al}
  \ee
  and, apart from that,
  \bea{35.10}
	\sqrt{\frac{\xi+1}{1+\eps\xi}} \cdot \frac{\al(1-\al)(\xi+1)^{\al-2} e^{(\xi+1)^\al}}{\sqrt{\chi(\xi)}}
	&=& \al(1-\al)\cdot \frac{(\xi+1)^{\al-\frac{3}{2}}}{\sqrt{1+\eps\xi}} \cdot e^{\frac{1}{2}(\xi+1)^\al} \nn\\
	&\le& \al\cdot (\xi+1)^\al \cdot e^{\frac{1}{2}(\xi+1)^\al} \nn\\
	&\le& \al\cdot\frac{2}{\al} \cdot e^\frac{1}{\al}
	= 2 e^\frac{1}{\al}
	\qquad \mbox{for all } \eps\in (0,1).
  \eea
  Now (\ref{35.9}) together with (\ref{35.7}) shows that
  \bas
	\chi''(\xi)
	\ge \al^2(\xi+1)^{2\al-2} e^{(\xi+1)^\al} - 2e^\frac{2}{\al}
	\qquad \mbox{for all $\xi\ge 0$ such that $(\xi+1)^\al<\frac{2}{\al}$,}
  \eas
  which combined with (\ref{35.8}) establishes (\ref{35.3}).\abs
  Apart from that, (\ref{35.10}) along with (\ref{35.8}) and (\ref{35.7}) implies that whenever $\xi\ge 0$ is such that
  $\chi''(\xi)\le 0$,
  \be{35.11}
	\frac{\xi+1}{1+\eps\xi} \cdot \frac{\chi''^2(\xi)}{\chi(\xi)}
	\le 4 e^\frac{2}{\al}
	\qquad \mbox{for all } \eps\in (0,1),
  \ee
  while within $\{\chi''>0\}$ it follows from (\ref{35.7}) and the inequality $\al\le 1$ that
  \bea{35.12}
	\frac{\xi+1}{1+\eps\xi} \cdot \frac{\chi''^2(\xi)}{\chi(\xi)}
	&\le& \frac{\xi+1}{1+\eps\xi} \cdot \frac{\big\{ \al^2 (\xi+1)^{2\al-2} e^{(\xi+1)^\al} \big\}^2}{e^{(\xi+1)^\al}} \nn\\
	&=& \al^4\cdot \frac{(\xi+1)^{4\al-3}}{1+\eps\xi} \cdot e^{(\xi+1)^\al} \nn\\
	&\le& \al^2 \cdot \frac{(\xi+1)^{2\al-1}}{1+\eps\xi} \cdot e^{(\xi+1)^\al}
	\qquad \mbox{for all } \eps\in (0,1).
  \eea
  As clearly
  \bas
	\frac{\xi+1}{1+\eps\xi} \cdot \frac {\chi'^2(\xi)}{\chi(\xi)}
	= \al^2 \cdot \frac{(\xi+1)^{2\al-1}}{1+\eps\xi} \cdot e^{(\xi+1)^\al}
	\qquad \mbox{for all $\xi\ge 0$ and } \eps\in (0,1)
  \eas
  by (\ref{35.2}), from (\ref{35.11}) and (\ref{35.12}) we infer (\ref{35.4}) for arbitrary $\eps\in (0,1)$.\abs
  Finally, given $\xi\ge 0$ we obtain from (\ref{35.1}) that if $(\xi+1)^\al\ge\frac{1}{\al}$, then
  \bas
	\frac{\frac{\xi+1}{1+\eps\xi} \cdot \chi(\xi)}{\al\cdot\frac{(\xi+1)^{\al+1}}{1+\eps\xi}\cdot e^{(\xi+1)^\al}}
	= \frac{1}{\al} \cdot \frac{1}{(\xi+1)^\al}
	\le 1
	\qquad \mbox{for all } \eps\in (0,1),
  \eas
  while if $(\xi+1)^\al < \frac{1}{\al}$, then
  \bas
	\frac{\xi+1}{1+\eps\xi} \cdot\chi(\xi)
	\le (\xi+1) e^{(\xi+1)^\al}
	\le \Big(\frac{1}{\al}\Big)^\frac{1}{\al} \cdot e^\frac{1}{\al}
	\qquad \mbox{for all } \eps\in (0,1).
  \eas
  This confirms (\ref{35.5}), whereas (\ref{35.6}) can directly be derived from (\ref{35.2}) by trivially estimating
  $(\xi+1)^\al \le (\xi+1)^{\al+1}$ for $\xi\ge 0$.
\qed
Indeed, when spelt out for functions of this form, Lemma \ref{lem11} leads to the main result of this section:
\begin{lem}\label{lem36}
  Given $K>0$, let $\al_0(K)$, $\gamma(K)$ and $\Gamma(K)$ be as in Lemma \ref{lem11},
  and let $\al\in (0,\al_0(K)]$. Then there exists $\Gamma_2(\al,K)>0$ such that
  if (\ref{init}) and (\ref{K}) hold, it follows that
  \bea{36.1}
	& & \hs{-20mm}
	\frac{d}{dt} \io \veps^2 e^{(\weps+1)^\al}
	+ \gamma(K) \io \frac{\weps+1}{1+\eps\weps} \cdot e^{(\weps+1)^\al} |\na\veps|^2
	+ 2 \io \veps^2 e^{(\weps+1)^\al} \nn\\
	&\le& 2 \Gamma(K) \al^2 \io \frac{(\weps+1)^{2\al-1}}{1+\eps\weps} \cdot e^{(\weps+1)^\al} |\na\weps|^2
	+ \Gamma(K) \al \io \frac{(\weps+1)^{\al+1}}{1+\eps\weps} \cdot e^{(\weps+1)^\al} \nn\\
	& & + \Gamma_2(\al,K) \io |\na\weps|^2
	+ \Gamma_2(\al,K)
  \eea
  for all $t>0$ and $\eps\in (0,1)$.
\end{lem}
\proof
  We let $\chi$ be as defined in Lemma \ref{lem35} and note that then, by (\ref{35.2}),
  \bas
	0 \le\chi'(\xi) \le \al(\xi+1)^{\al-1} e^{(\xi+1)^\al} \le \al e^{(\xi+1)^\al} =\al\chi(\xi) \le \al_0(K)\chi(\xi)
	\qquad \mbox{for all } \xi\ge 0,
  \eas
  so that since additionally $\al_0(K)\le 1$, we may combine Lemma \ref{lem11} with (\ref{35.1}) to see that
  \bea{36.2}
	& & \hs{-20mm}
	\frac{d}{dt} \io \veps^2 e^{(\weps+1)^\al}
	+ \gamma(K) \io \frac{\weps+1}{1+\eps\weps} \cdot e^{(\weps+1)^\al} |\na\veps|^2
	+ 2 \io \veps^2 e^{(\weps+1)^\al} \nn\\
	&\le& \Gamma(K) \io \frac{\weps+1}{1+\eps\weps} \cdot \frac{\chi'^2(\weps)+\chi''^2(\weps)}{\chi(\weps)} \cdot |\na\weps|^2
	+ \Gamma(K) \int_{\{\chi''(\weps<0\}} |\chi''(\weps)| \cdot |\na\weps|^2 \nn\\
	& & + \Gamma(K) \io \frac{\weps+1}{1+\eps\weps} \cdot \chi(\weps)
	\qquad \mbox{for all $t>0$ and } \eps\in (0,1).
  \eea
  Here, (\ref{35.4}) ensures that
  \bea{36.3}
	\Gamma(K) \io \frac{\weps+1}{1+\eps\weps} \cdot \frac{\chi'^2(\weps)+\chi''^2(\weps)}{\chi(\weps} \cdot |\na\weps|^2
	&\le& 2\Gamma(K) \al^2 \io \frac{(\weps+1)^{2\al-1}}{1+\eps\weps} \cdot e^{(\weps+1)^\al} |\na\weps|^2 \nn\\
	& & + 4e^\frac{2}{\al} \Gamma(K) \io |\na\weps|^2
  \eea
  for all $t>0$ and $\eps\in (0,1)$,
  while from (\ref{35.3}) we know that
  \be{36.4}
	\Gamma(K) \int_{\{\chi''(\weps<0\}} |\chi''(\weps)| \cdot |\na\weps|^2
	\le 4e^\frac{4}{\al} \Gamma(K) \io |\na\weps|^2
	\qquad \mbox{for all $t>0$ and } \eps\in (0,1).
  \ee
  Since (\ref{35.5}) warrants that
  \bas
	\Gamma(K) \io \frac{\weps+1}{1+\eps\weps} \cdot \chi(\weps)
	\le \Gamma(K) \al \io \frac{(\weps+1)^{\al+1}}{1+\eps\weps} \cdot e^{(\weps+1)^\al}
	+ \Big(\frac{e}{\al}\Big)^\frac{1}{\al} \Gamma(K) |\Om|
  \eas
  for all $t>0$ and $\eps\in (0,1)$,
  a combination of (\ref{36.2}) with (\ref{36.3}) and (\ref{36.4}) leads to (\ref{36.1}) with		
  $\Gamma_2(\al,K):=
  \max\Big\{ 4(e^\frac{2}{\al}+ e^\frac{4}{\al}) \Gamma(K) \, , \, (\frac{e}{\al})^\frac{1}{\al} \Gamma(K) |\Om| \Big\}$.
\qed
\subsection{Estimating $\io e^{u^\al}$ for small $\al$. Conclusion}
We are thus prepared to establish an approximate version of our main estimate announced in Theorem \ref{theo14},
which indeed can be obtained by combining Lemma \ref{lem34} with Lemma \ref{lem36}, and estimating 
the second to last summand in (\ref{34.1}) by means of the interpolation inequality from Lemma \ref{lem32}.
\begin{lem}\label{lem37}
  Let $K>0$. Then there exist $\al=\al(K)>0$, $C(K)>0$ and $\eps_0=\eps_0(K)\in (0,1)$ such that if (\ref{init}) and (\ref{K}) hold,
  then it follows that
  \be{13.1}
	\io e^{\ueps^\al(\cdot,t)} \le C(K)
	\qquad \mbox{for all $t>0$ and } \eps\in (0,\eps_0).
  \ee
\end{lem}
\proof
  We fix $K>0$ and let $\al_0=\al_0(K), \gamma=\gamma(K)$ and $\Gamma=\Gamma(K)$ from Lemma \ref{lem11}, and abbreviating
  \be{37.02}
	c_1\equiv c_1(K):=\Lam_3\big( 2K|\Om| \big)
  \ee
  with $\Lam_3(\cdot)$ as provided by Lemma \ref{lem32}, we set
  \be{37.2}
	b\equiv b(K):=\frac{16\Gamma}{D},
  \ee
  choose $\al=\al(K)\in (0,\min\{1,\frac{2}{n}\})$ small enough fulfilling
  \be{37.3}
	\al^2 \le \frac{2D\gamma}{3b(d-D)^2}
  \ee
  as well as
  \be{37.4}
	\al \le \frac{bD}{8c_1\cdot (3b+\Gamma)},
  \ee
  and fix $\eps_0=\eps_0(K)\in (0,1)$ in such a way that $\eps_0^\al \le \frac{\al}{2}$.
  Taking $\Gamma_2=\Gamma_2(\al,K)$ from Lemma \ref{lem36} and assuming (\ref{init}) as well as (\ref{K}),
  we may then invoke Lemma \ref{lem34} along with Lemma \ref{lem36} to find that
  \bas
	& & \hs{-30mm}
	\frac{d}{dt} \bigg\{ b \io \frac{\weps+1}{1+\eps\weps} \cdot e^{(\weps+1)^\al}
	+ \io \veps^2 e^{(\weps+1)^\al} \bigg\}
	+ b \io \frac{\weps+1}{1+\eps\weps} \cdot e^{(\weps+1)^\al}
	+ 2 \io \veps^2 e^{(\weps+1)^\al} \nn\\
	& & + \frac{b D \al^2}{4} \io \frac{(\weps+1)^{2\al-1}}{1+\eps\weps} \cdot e^{(\weps+1)^\al} |\na\weps|^2
	+ \gamma \io \frac{\weps+1}{1+\eps\weps} \cdot e^{(\weps+1)^\al} |\na\veps|^2 \nn\\
	&\le& \frac{3b(d-D)^2 \al^2}{2D} \io \frac{(\weps+1)^{2\al-1}}{1+\eps\weps} \cdot e^{(\weps+1)^\al} |\na\veps|^2
	+ \frac{b(d-D)^2 e^\frac{1}{\al}}{2D} \io |\na\veps|^2 \nn\\
	& & + 3b\al \io \frac{(\weps+1)^{\al+1}}{1+\eps\weps} \cdot e^{(\weps+1)^\al}
	+ 2b\Big(\frac{e}{\al}\Big)^\frac{1}{\al} \cdot |\Om| \nn\\
	& & + 2\Gamma \al^2 \io \frac{(\weps+1)^{2\al-1}}{1+\eps\weps} \cdot e^{(\weps+1)^\al} |\na\weps|^2
	+ \Gamma \al \io \frac{(\weps+1)^{\al+1}}{1+\eps\weps} \cdot e^{(\weps+1)^\al} \nn\\
	& & + \Gamma_2 \io |\na\weps|^2
	+ \Gamma_2
	\qquad \mbox{for all $t>0$ and } \eps\in (0,\eps_0),
  \eas
  where we note that
  \bas
	\frac{bD\al^2}{4} - 2\Gamma\al^2 = \frac{bD\al^2}{8}
  \eas
  by (\ref{37.2}), and that
  \bas
	& & \hs{-20mm}
	\frac{3b(d-D)^2 \al^2}{2D} \io \frac{(\weps+1)^{2\al-1}}{1+\eps\weps} \cdot e^{(\weps+1)^\al} |\na\veps|^2 \nn\\
	&\le& \frac{3b(d-D)^2 \al^2}{2D} \io \frac{\weps+1}{1+\eps\weps} \cdot e^{(\weps+1)^\al} |\na\veps|^2 \\
	&\le& \gamma \io \frac{\weps+1}{1+\eps\weps} \cdot e^{(\weps+1)^\al} |\na\veps|^2
	\qquad \mbox{for all $t>0$ and } \eps\in (0,\eps_0)
  \eas
  according to (\ref{37.3}) and the fact that $\al\le 1$.
  Rearranging and trivially estimating $2\io \veps^2 e^{(\weps+1)^\al} \ge \io \veps^2 e^{(\weps+1)^\al}$
  for $t>0$ and $\eps\in (0,\eps_0)$, we thus infer that for
  \bas
	\yeps(t):=b \io \frac{\weps+1}{1+\eps\weps} \cdot e^{(\weps+1)^\al}
	+ \io \veps^2 e^{(\weps+1)^\al},
	\qquad t\ge 0, \ \eps\in (0,\eps_0),
  \eas
  we have
  \bea{37.5}
	& & \hs{-30mm}
	\yeps'(t) + \yeps(t)
	+ \frac{bD\al^2}{8} \io \frac{(\weps+1)^{2\al-1}}{1+\eps\weps} \cdot e^{(\weps+1)^\al} |\na\weps|^2 \nn\\
	&\le& (3b+\Gamma)\al \io \frac{(\weps+1)^{\al+1}}{1+\eps\weps} \cdot e^{(\weps+1)^\al} \nn\\
	& & + c_2 \io |\na\veps|^2 + c_3 \io |\na\weps|^2 + c_4
	\qquad \mbox{for all $t>0$ and } \eps\in (0,\eps_0)
  \eea
  with $c_2\equiv c_2(K):=\frac{b(d-D)^2 e^\frac{1}{\al}}{2D}$,
  $c_3\equiv c_3(K):=\Gamma_2$ and
  $c_4\equiv c_4(K):=2b(\frac{e}{\al})^\frac{1}{\al} \cdot |\Om| + \Gamma_2$.\abs
  At this point, based on our restriction that $\al\le \min\{1,\frac{2}{n}\}$
  the interpolation result from Lemma \ref{lem32} applies so as to ensure that, again since
  \bas
	\io \weps = \io \ueps + \io \veps
	\le \io u_0 + \max \bigg\{ \io u_0, \io v_0\bigg\}
	\le 2K|\Om|
	\qquad \mbox{for all $t>0$ and } \eps\in (0,\eps_0)
  \eas
  by (\ref{mass}) and (\ref{K}), with $c_1$ as in (\ref{37.02}) and with $c_5\equiv c_5(K):=\Lam_4(\al,2K|\Om|)$,
  $\Lam_4(\cdot,\cdot)$ being taken from Lemma \ref{lem32}, we have
  \bas
	& & \hs{-20mm}
	(3b+\Gamma) \al \io \frac{(\weps+1)^{\al+1}}{1+\eps\weps} \cdot e^{(\weps+1)^\al} \nn\\
	&\le& (3b+\Gamma) \al \cdot c_1\al^2 \io \frac{(\weps+1)^{2\al-1}}{1+\eps\weps} \cdot e^{(\weps+1)^\al} |\na\weps|^2 \nn\\
	& & + (3b+\Gamma) \al \cdot c_5 \io |\na\weps|^2
	+ (3b+\Gamma) \al \cdot c_5
	\qquad \mbox{for all $t>0$ and } \eps\in (0,\eps_0).
  \eas
  As our smallness condition in (\ref{37.4}) guarantees that
  \bas
	(3b+\Gamma) \al\cdot c_1\al^2
	\le \frac{bD\al^2}{8},
  \eas
  this implies that (\ref{37.5}) entails the inequality
  \bas
	\yeps'(t) + \yeps(t)
	\le \heps(t):=c_2 \io |\na\veps|^2
	+ \big\{ c_3 + (3b+\Gamma) c_5\al\big\} \cdot \io |\na\weps|^2
	+ c_4 + (3b+\Gamma) c_5\al
  \eas
  for all $t>0$ and $\eps\in (0,1)$.
  Since from Lemma \ref{lem4} we know that with some $c_6=c_6(K)>0$ we have
  \bas
	\int_t^{t+1} \heps(s) ds \le c_6
	\qquad \mbox{for all $t>0$ and } \eps\in (0,\eps_0),
  \eas
  and since thus
  \bas
	\int_0^t e^{-(t-s)} \heps(s) ds
	\le \frac{c_6}{1-e^{-1}}
	\qquad \mbox{for all $t>0$ and } \eps\in (0,\eps_0)
  \eas
  according to an elementary inequality recorded in \cite[Lemma 3.4]{win_JFA}, this shows that
  \bas
	\yeps(t)
	&\le& \yeps(0) e^{-t} + \int_0^t e^{-(t-s)} \heps(s) ds \\
	&\le& b \io (u_0+v_0+1) e^{(u_0+v_0+1)^\al}
	+ \io v_0^2 e^{(u_0+v_0+1)^\al}
	+ \frac{c_6}{1-e^{-1}} \\
	&\le& b (2K+1) e^{(2K+1)^\al} \cdot |\Om|
	+ K^2 e^{(2K+1)^\al} \cdot |\Om|
	+ \frac{c_6}{1-e^{-1}}
	\qquad \mbox{for all $t>0$ and } \eps\in (0,\eps_0).
  \eas
  Observing that $\frac{\xi+1}{1+\eps\xi} \ge \frac{\eps\xi+1}{1+\eps\xi}=1$ for all $\xi\ge 0$ and $\eps\in (0,\eps_0)$ and hence
  \bas
	\yeps(t)\ge b\io e^{(\weps+1)^\al} \ge b \io e^{\ueps^\al}
	\qquad \mbox{for all $t>0$ and } \eps\in (0,\eps_0),
  \eas
  we may conclude as intended.
\qed
In its essence, our main result has thereby actually been accomplished already:\abs
\proofc of Theorem \ref{theo14}. \quad
  Given $K>0$, from Lemma \ref{lem37} and Lemma \ref{lem5} we know that there exist $c_1=c_1(K)>0$ and $c_2=c_2(K)>0$ such that
  whenever (\ref{init}) and (\ref{K}) hold, the solutions $(\ueps,\veps)$ of (\ref{0eps}) from Lemma \ref{lem_loc} satisfy
  \be{14.3}
	\io e^{\ueps^\al(\cdot,t)} \le c_1
	\quad \mbox{and} \quad
	\|\veps(\cdot,t)\|_{L^\infty(\Om)} \le c_2
	\qquad \mbox{for all $t>0$ and } \eps\in (0,\eps_0),
  \ee
  where $\al=\al(K)$ and $\eps_0=\eps_0(K)$ are as determined by Lemma \ref{lem37}.
  Apart from that, in view of (\ref{82.2}), (\ref{82.4}) and the Fubini-Tonelli theorem, there exists a null set $N\subset (0,\infty)$
  such that with $(u,v)$ and $(\eps_j)_{j\in\N}$ as provided by Lemma \ref{lem82} we have
  \bas
	\ueps(\cdot,t) \to u(\cdot,t)
	\quad \mbox{and} \quad
	\veps(\cdot,t) \to v(\cdot,t)
	\quad \mbox{a.e.~in } \Om
	\qquad \mbox{for all } t\in (0,\infty)\sm N
  \eas
  as $\eps=\eps_j\searrow 0$. By utilizing Fatou's lemma, from (\ref{14.3}) we thus infer that
  \bas
	\io e^{u^\al(\cdot,t)} \le c_1
	\quad \mbox{and} \quad
	\|v(\cdot,t)\|_{L^\infty(\Om)} \le c_2
	\qquad \mbox{for all } t\in (0,\infty) \sm N,
  \eas
  so that the claim results upon recalling from Lemma \ref{lem82} that $(u,v)$ indeed is a global weak solution of (\ref{0})
  in the sense specified in Theorem \ref{theo14}.
\qed
\mysection{Appendix: Proof of Lemma \ref{lem7}}
As annonced, let us finally describe how the $\eps$-dependent $W^{1,p}$ bounds claimed in Lemma \ref{lem7} can be derived from
Lemma \ref{lem6} and Lemma \ref{lem60}.\abs
\proofc of Lemma \ref{lem7}.\quad
  As an argument addressing a closely related situation can be found detailed in \cite[Lemmata 5.2-5.6]{taowin_JMPA},
  we may confine ourselves here with an outline of the main steps.\abs
  \underline{Step 1}: {\em Deriving an enery-type inequality for $\io |\na\veps|^p$ with $p\ge 4$.} \quad
  Using the identities
  \be{7.11}
	\na \ueps =\na\weps-\na \veps \quad\mbox{and}\quad \na \frac{\ueps}{1+\eps\ueps}=\frac{1}{(1+\eps\ueps)^2} \na \ueps
  \ee
  as well as the inequality
  \be{7.12}
	\Big|\frac{\ueps}{1+\eps\ueps}\Big| \le \frac{1}{\eps}
  \ee
  as seen in \cite[Lemma 5.3]{taowin_JMPA} we can obtain that for each $p\ge 4$ and any $\sigma>0$
  one can find $K_1(\sigma, \eps, p, u_0, v_0)>0$ satisfying
  \be{7.13}
	\frac{1}{p}\frac{d}{dt} \io |\na \veps|^p +\frac{d}{16} \io |\na \veps|^{p-2}|D^2 \veps|^2
	\le \sigma \io |\na \weps|^{p+2} +K_1\io|\na \veps|^{p+2}+K_1
  \ee
  for all $t\in (0,\tme)$.\abs
  \underline{Step 2}: {\em Establishing an enery-type inequality for $\io |\na \weps|^p$ with $p\ge 4$.} \quad
  Relying on (\ref{0w}) and (\ref{7.11}) and following \cite[Lemma 5.2]{taowin_JMPA}, we can show that for all $p\ge 4$
  there exists $K_2(\eps, p, u_0, v_0)>0$ satisfying
  \be{7.14}
	\hspace*{-2mm}
	\frac{1}{p}\frac{d}{dt} \io |\na \weps|^p +\frac{D}{2} \io |\na \weps|^{p-2}|D^2 \weps|^2
	\le K_2 \io |\na \weps|^{p-2}|D^2 \veps|^2 +K_2\io|\na \weps|^{p} +K_2\io|\na \veps|^{p}
 \ee
  for all $t\in (0,\tme)$.\abs
\underline{Step 3}: {\em Studying the evolution of the coupled-gradient functional $\io |\na\veps|^2 |\na \weps|^{p-2}$ for $p\ge 6$.}
\quad
  To control the first integral on the right-hand side of (\ref{7.14}), using (\ref{7.11}) and (\ref{7.12})
  we see that for each $p\ge 6$ and any $\eta>0$ there exists $K_3(\eta, \eps, p, u_0, v_0)>0$ fulfilling
  \bea{7.15}
	\frac{d}{dt} \io |\na\veps|^2 |\na \weps|^{p-2}
 	&+& \frac{d}{4}\io |\na \weps|^{p-2}|D^2 \veps|^2 \nn\\
 	&\le& \eta \io |\na \weps|^2 |D^2\weps|^2 + \eta\io |\na\weps|^{p+2}\nn\\
 	& & +K_3 \io |\na\veps|^{p-2}|D^2\veps|^2 +K_3\io |\na\veps|^{p+2} +K_3
  \eea
  for all $t\in (0,\tme)$ (cf.~\cite[Lemma 5.4]{taowin_JMPA}).\abs
\underline{Step 4}: {\em Recalling two useful interpolation inequalities.} \quad
  In order to expediently deal with the integrals $\io |\na \weps|^{p+2}$ and $\io |\na \veps|^{p+2}$
  appearing on the right-hand sides of (\ref{7.13})-(\ref{7.15}), we shall invoke
  the following two interpolation properties (cf. \cite[Lemma 5.5]{taowin_JMPA}):\\
  i) \ Given any $p\ge 2$, one can find $K_{41}>0$ such that whenever $\vp\in C^2(\bom)$ satisfies $\frac{\pa\vp}{\pa\nu}=0$
  on $\pO$, we have
  \be{7.16}
	\io |\na\vp|^{p+2}
	\le K_{41} \cdot \bigg\{ \io |\na \vp|^{p-2} |D^2 \vp|^2 \bigg\} \cdot \|\vp\|_{L^\infty(\Om)}^2.
  \ee
  ii) \ Let $\omega:(0,\infty)\to (0,\infty)$ be nondecreasing, and let $p\ge 2$ and $\tilde{\eta}>0$.
  Then there exists $K_{42}(\tilde{\eta},q,\omega)>0$ such that if $\vp\in C^2(\bom)$
  is such that $\frac{\pa\vp}{\pa\nu}=0$ on $\pO$ and that
  \bas
	\qquad \mbox{for each $\del>0$ and any $x\in\bom$ and $y\in\bom$ fulfilling $|x-y|<\omega(\del)$, we have }
	|\vp(x)-\vp(y)| < \del,
  \eas
  it follows that
  \be{7.17}
    	\io |\na \vp|^{p+2}
	\le \tilde{\eta} \io |\na\vp|^{p-2} |D^2\vp|^2
	+ K_{42}(\tilde{\eta},p,\omega) \|\vp\|_{L^\infty(\Om)}^{p+2}.
  \ee
\underline{Step 5}: {\em Completing the proof.} \quad
  To compensate the summands appearing on the right-hand sides of (\ref{7.13})-(\ref{7.15})
  by means of the diffusion-related integrals on the left-hand sides therein,
  we need to suitably select the two free small paramters $\eta$ and $\sigma$ in (\ref{7.15}) and (\ref{7.13})
  and design an appropriate linear combination of the functionals $\io |\na\veps|^p, \io |\na \weps|^p$ and
  $\io |\na\veps|^2 |\na \weps|^{p-2}$. For this purpose, we first invoke (\ref{7.16}) in conjunction with
  Lemma \ref{lem60} and Lemma \ref{lem5} to fix $c_1\equiv c_1(\eps, p, u_0, v_0)$ such that
  \be{7.18}
	\io |\na \weps|^{p+2} \le c_1 \io |\na \weps|^{p-2} |D^2\weps|^2
	\qquad\mbox{for all $t\in (0,\tme)$ and $\eps\in (0,1)$},
  \ee
  and taking $K_2 =K_2(\eps, p, u_0, v_0)$ as obtained in (\ref{7.14}), we let
  \be{7.19}
	\beta=\beta(\eps, p, u_0, v_0) :=\frac{d}{4K_2}
  \ee
  as well as
  \be{7.110}
	\eta =\eta(\eps, p, u_0, v_0) :=\min\Big\{\frac{\beta D}{4}, \, \, \frac{\beta D}{8c_1}\Big\}.
  \ee
  Thereupon fixing $K_3=K_3(\eta, \eps, p, u_0, v_0)$ such that (\ref{7.15}) holds, we take
  \be{7.111}
	b\equiv b(\eps, p, u_0, v_0) :=\frac{d}{32 K_3}
  \ee
  and
  \be{7.112}
	\sigma\equiv \sigma(\eps, p, u_0, v_0) :=\frac{\beta b D}{16 c_1}
  \ee
  and let $K_1=K_1(\sigma, \eps, p, u_0, v_0)$ be as accordingly be introduced near (\ref{7.13}).
  Now defining
  \be{7.113}
	y_\eps(t)
	:=\frac{1}{p} \io |\na \veps(\cdot, t)|^p +b\io |\na \veps(\cdot, t)|^2 |\na \weps(\cdot, t)|^{p-2}
	+\frac{\beta b}{p}\io |\na \weps(\cdot, t)|^p,
	\qquad t\in [0,\tme),
  \ee
  by straightforward computation of (\ref{7.13})-(\ref{7.15}) with (\ref{7.19})-(\ref{7.113}) we obtain that
  \bea{7.114}
	\hspace*{-6mm}
	y_\eps'(t)
	&=& -\Big(\frac{d}{16}-bK_3\Big) \io |\na \veps|^{p-2} |D^2\veps|^2
    	-\Big(\frac{bd}{4}-\beta b K_2\Big) \io |\na\weps|^{p-2} |D^2 \veps|^2\nn\\
	& & -\Big(\frac{\beta bD}{2}-b\eta\Big)\io |\na \weps|^{p-2} |D^2\weps|^2\nn\\
	& & +(K_1+bK_3) \io |\na\veps|^{p+2} +(\sigma +b\eta)\io |\na\weps|^{p+2}\nn\\
	& & + \beta b K_2 \io |\na\veps|^p
    	+ \beta b K_2 \io |\na\weps|^p  \nn\\
	& & +K_1 +bK_3\nn\\
	&\le& -\frac{d}{32} \io |\na \veps|^{p-2} |D^2\veps|^2
       -\frac{\beta b D}{4} \io |\na \weps|^{p-2} |D^2\weps|^2
        \nn\\[1mm]
	& & +I
   	\qquad\mbox{for all $t\in (0,\tme)$}
  \eea
  due to the fact that $\frac{d}{16}-bK_3=\frac{d}{32}, \frac{bd}{4}-\beta b K_2=0$ and
  $\frac{\beta bD}{2}-b\eta\ge \frac{\beta bD}{4}$ by (\ref{7.111}), (\ref{7.19}) and the first
  restriction in (\ref{7.110}), respectively, where for $t\in (0,\tme)$ we have set
  \bas
	\hspace{-10mm}
	I
	&:=& (K_1+bK_3) \io |\na\veps|^{p+2} +(\sigma +b\eta)\io |\na\weps|^{p+2}\nn\\
	& & + \beta b K_2 \io |\na\veps|^p
    	+ \beta b K_2 \io |\na\weps|^p
	+K_1 +bK_3.
  \eas
  Here, Young's inequality entails that for all $t\in (0,\tme)$,
  \be{7.116}
	\beta b K_2 \io |\na\veps|^p
	\le \beta b K_2 \io |\na\veps|^{p+2}
    	+ \beta b K_2 |\Om|
  \ee
  and
  \bea{7.117}
	\beta b K_2 \io |\na\weps|^p
	&=& \io \Big(\sigma |\na\weps|^{p+2}\Big)^\frac{p}{p+2}
    	\cdot \sigma^{-\frac{p}{p+2}} \beta b K_2 \nn\\
	&\le& \sigma \io |\na\weps|^{p+2} +c_2
  \eea
  with $c_2\equiv c_2(\eps, p, u_0, v_0):=\sigma^{-\frac{p}{2}}\big(\beta b K_2\big)^\frac{p+2}{2}$,
  and from (\ref{7.18}) we obtain that
  \be{7.118}
	(2\sigma +b\eta)\io |\na\weps|^{p+2} \le (2\sigma +b\eta) c_1\io |\na\weps|^{p-2}|D^2\weps|^2
	\qquad \mbox{for all } t\in (0,\tme).
  \ee
  Collecting (\ref{7.116})-(\ref{7.118}) leads to the inequality
  \be{7.119}
	I \le \Big(K_1+bK_3+\beta bK_2 +\frac{1}{p}+b\Big) \io |\na\veps|^{p+2}
 	+ (2\sigma +b\eta) c_1\io |\na\weps|^{p-2}|D^2\weps|^2 +c_3
	\qquad \mbox{for all } t\in (0,\tme)
  \ee
  with $c_3\equiv c_3(\eps, p, u_0, v_0):=K_1+bK_3 + \beta b K_2 |\Om| + c_2$,
  and inserting this into (\ref{7.114}) we arrive at the inequality
  \bea{7.120}
	y_\eps'(t)
	+ \frac{d}{32} \io |\na \veps|^{p-2} |D^2\veps|^2
	&\le& -\Big\{\frac{\beta bD}{4} -(2\sigma +b\eta)c_1\Big\} \io |\na \weps|^{p-2} |D^2\weps|^2\nn\\
	& & + (K_1+bK_3+\beta bK_2) \io |\na\veps|^{p+2} +c_3\nn\\
	&\le & (K_1+bK_3+\beta bK_2) \io |\na\veps|^{p+2} +c_3
  \eea
  for all $t\in (0,\tme)$, because $\frac{\beta bD}{4} -(2\sigma +b\eta)c_1
  =\big(\frac{\beta bD}{8} -2\sigma c_1\big) + \big(\frac{\beta bD}{8} -b\eta c_1\big)
  =\frac{\beta bD}{8} -b\eta c_1 \ge 0$ according to (\ref{7.112}) and the second restriction in (\ref{7.110}).
  Now in line with Lemma \ref{lem6}, we may apply (\ref{7.17}) to $\tilde{\eta}:=\frac{1}{K_1+bK_3+\beta bK_2}\cdot \frac{d}{32}$
  to find $c_4=c_4(\eps, p, u_0, v_0)>0$ such that
  \bas
 	(K_1+bK_3+\beta bK_2) \io |\na\veps|^{p+2}
 	\le \frac{d}{32} \io |\na \veps|^{p-2} |D^2\veps|^2 +c_4,
  \eas
  so that (\ref{7.120}) implies that
  \bas
	y_\eps'(t) \le c_3+c_4
	\qquad\mbox{for all $t\in (0,\tme)$}
  \eas
  and thus $y_\eps(t) \le y_\eps'(0) \cdot (c_3+c_4) \tme$ for all $t\in (0,\tme)$.
  Since $|\na \ueps(\cdot, t)|^p \le 2^{p-1} (|\na \weps(\cdot, t)|^p + |\na \veps(\cdot, t)|^p)$, this leads to (\ref{7.1}).
\qed

\bigskip

{\bf Acknowledgement.} \quad
  The first author was supported by the {\em National Natural Science Foundation of China
   (No. 12571222)}.  The second author  acknowledges support of the
  {\em Deutsche Forschungsgemeinschaft} (Project No.~462888149).
\bigskip
\end{document}